\documentclass[a4paper,11pt,reqno]{amsart}

\usepackage[normalem]{ulem}
\usepackage{color}
\usepackage{amsmath}
\usepackage{amssymb}
\usepackage{array}
\usepackage{booktabs}
\usepackage{amsthm}
\usepackage[hidelinks]{hyperref}

\usepackage{enumitem}

\theoremstyle{plain}
\newtheorem{thm}{Theorem}[section]
\newtheorem{prop}[thm]{Proposition}
\newtheorem{cor}[thm]{Corollary}
\newtheorem{lem}[thm]{Lemma}

\theoremstyle{definition}
\newtheorem{defn}[thm]{Definition}

\theoremstyle{remark}
\newtheorem{rem}[thm]{Remark}

\newtheorem{introthm}{\bf Theorem}

\usepackage{setspace}
\newcommand{\m}{\mathfrak{m}}
\newcommand{\p}{\partial}

\newcommand{\R}{\mathbb{R}}

\newcommand{\ric}{\mathrm{Ric}}
\newcommand{\trace}{\mathrm{tr}}

\newcommand{\dive}{\mathrm{div}}
\def\tr{\textmd{tr}}
\newcommand{\btr}[1]{\left\vert#1\right\vert}
\newcommand{\Sbb}{\mathbb{S}}

\renewcommand{\d}{\,\mathrm{d}}

\renewcommand{\S}{\Sigma}

\newcounter{mnotecount}[section]

\numberwithin{equation}{section}

\title[]{Extensions of spacetime Bartnik data and estimates for the Bartnik mass outside of time-symmetry} 

\author[S. McCormick]{Stephen McCormick}
\address{Institutionen f\"or teknikvetenskap och matematik \\
	Lule{\aa} tekniska universitet \\
	971 87 Lule\aa \\
	Sweden} 
\email{stephen.mccormick@ltu.se}

\author[M. Wolff]{Markus Wolff}
\address{Institutionen f\"or Matematik \\
	Kungliga Tekniska H\"ogskolan \\
	100 44 Stockholm \\
	Sweden} 
\email{markuswo@kth.se}

\begin{document}
	
	\begin{abstract}
		Bartnik's quasi-local mass is a functional on Bartnik data $(\mathbb S^2,\gamma,H,P,\omega^\perp)$, consisting of a metric $\gamma$, scalar functions $H$ and $P$, and a 1-form $\omega^\perp$ on the $2$-sphere $\mathbb S^2$. We construct initial data $(M,g,K)$ for the Einstein equations with boundary $\Sigma\cong\mathbb S^2$, and boundary conditions for $g$ and $K$ determined by Bartnik data with $H,P$ constant and $\omega^\perp\equiv0$. Furthermore this initial data agrees with spherically symmetric initial data for a Schwarzschild spacetime outside of a compact set with controlled mass. As an application, we obtain estimates for the Bartnik mass for such Bartnik data, outside of the time-symmetric setting.
		
		We also construct initial data on the cylinder $\mathbb S^2\times[0,1]$ connecting this same class of Bartnik data to time-symmetric data so that estimates for the Bartnik mass outside of time-symmetry can be obtained from prior estimates for time-symmetric data.
	\end{abstract}
	
\maketitle

\section{Introduction}
Initial data $(M,g,K)$ for the Einstein equations consists of a Riemannian $3$-manifold $(M,g)$ equipped with a symmetric $2$-tensor $K$, which is to be understood as the second fundamental form of $(M,g)$ embedded in a $4$-dimensional Lorentzian manifold satisfying the Einstein equations. Given a bounded domain $\Omega$ in such an initial data set, a quasi-local mass is a measure of the total mass or energy contained within $\Omega$. There are many competing definitions of a good quasi-local mass, each with advantages and disadvantages. This article concerns the definition due to Bartnik \cite{Bartnik1989-QLM-PRL}, which is sometimes said to be the most likely to give the correct physical measure of mass, if only it were possible to compute. In fact, the value of the Bartnik mass is only known in very special cases. It is essentially a localisation of the ADM mass, obtained by taking the infimum of the ADM mass over a space of admissible asymptotically flat initial data sets with boundary prescribed by data from $\Omega$. A precise definition is given in Section \ref{Ssec-Bartnik}. In recent years there have been many new developments in understanding the Bartnik mass, including estimates of its value (for example,  \cite{CabreraPacheco2017-AF-CMC-Bartnik,lin2016bartnik, MantoulidisSchoen2015-Apparent-Horizons, MiaoPiubello2024,MiaoWangXie2020}). However, the majority of work to-date on the topic has been under the assumption of time-symmetry, which is to say, $K\equiv 0$. This article relaxes that assumption.

We call a tuple $(\Sigma,\gamma,H,P,\omega^\perp)$ \textit{Bartnik data}, where $\Sigma$ is diffeomorphic to the $2$-sphere, $\gamma$ is a Riemannian metric of $\S$, $H$ and $P$ are smooth functions on $\Sigma$ and $\omega^\perp$ is a smooth one form on $\Sigma$. Note that $P$ and $\omega^\perp$ give boundary conditions for $K$, so it is common in the literature (where time-symmetric is often assumed implicitly) to refer to simply $(\Sigma,\gamma,H)$ as Bartnik data. The Bartnik data can be understood as prescribing the following boundary data on the boundary $\Sigma=\p M$ of an initial data set $(M,g,K)$: the induced metric $\gamma$, the mean curvature $H$, $P=\trace_\Sigma K$, and $\omega^\perp=K(\nu,\cdot)$ for unit normal $\nu$, where the normal direction here and which defines $H$ is inward relative to the Bartnik extension $(M,g,K)$ constructed.

In this article, we provide estimates for the Bartnik mass with the assumption of time-symmetry relaxed to include nonzero $P$. Specifically, we construct initial data sets $(M,g,K)$ with controlled mass and prescribed Bartnik data $(\Sigma,\gamma,H_o,P_o,0)$, for $H_o$ and $P_o$ constant with $H_o^2\ge P_o^2$, whose ADM mass gives an upper bound for the spacetime Bartnik mass. The assumption that $H_o$ and $P_o$ be constant is not strictly necessary for the construction of extensions, but is chosen rather for simplicity since constant data can be used to estimate non-constant data via the positive mass theorem with corners (see, for example, Remark 4.3 of \cite{mccormick2020gluing}). On the other hand, the additional restriction of $\omega^\perp=0$ is a genuine restriction on the data that we hope to overcome in future work (see Remark \ref{rem-nonzerow} for some comments on this case). With that said, in the case of marginally outer trapped surface (MOTS) initial data, it turns out that $\omega^\perp$ can be handled cleanly. In fact, in his thesis, Lin \cite{Lin2024-Spin-Creased} was able to obtain an upper bound for the Bartnik mass of a stable MOTS equal to the (conjectured) lower bound from the Penrose inequality directly analogous to the optimal upper bound for a stable minimal surface in the time-symmetric case due to Mantoulidis and Schoen \cite{MantoulidisSchoen2015-Apparent-Horizons}.

The extensions that we construct are in the same spirit as several others before us \cite{CabreraPacheco2017-AF-CMC-Bartnik,CabreraPacheco2018-AH-Bartnik,MantoulidisSchoen2015-Apparent-Horizons,MiaoPiubello2024,MiaoWangXie2020,miao2018compact}, by constructing ``collar" initial data on a cylinder $M=\Sigma\times[0,1]$ then gluing it to an exterior model solution. Specifically, we extend on the idea first developed by Mantoulidis and Schoen \cite{MantoulidisSchoen2015-Apparent-Horizons} and construct a collar that can be smoothly glued to a (non-time-symmetric) spherically symmetric initial data set in a Schwarzschild spacetime. The key difference here compared to previous work is that we must simultaneously handle both the metric and the tensor $K$ along the collar, ensuring the full dominant energy condition (DEC) is preserved, not only the non-negativity of the scalar curvature. Furthermore, some additional care must be taken to ensure that the extensions contain no MOTS enclosing the boundary prescribed by the Bartnik data.

We remark that the Schwarzschild initial data that our extensions agree with outside a compact set can be quite general. That is, although asymptotically flat or asymptotically hyperboloidal initial data is more directly applicable to the Bartnik mass, we are able to construct extensions that may have general asymptotic behavior in the Schwarzschild exterior region.

Using our general construction (Theorem \ref{thm_main1}), we find criteria for given Bartnik data to admit an admissible extension.

\begin{introthm}[Corollary \ref{kor_any}]\label{thm-intromain}
	Let $(\Sigma,\gamma,H_o,P_o,0)$ be Bartnik data with $\gamma$ having strictly positive scalar curvature, where $H_o$, $P_o$ are constant with $H_o\ge \btr{P_o}>0$, and Hawking mass $\m_o\ge 0$. There are constants $0<\beta\le 1$ and $\alpha \ge 0$ depending on $\gamma$ such that if
	\begin{align}\label{eq-introAcond}
	H_o^2&<\frac{4}{r_o^2}\frac{\beta}{\alpha}
	\end{align}
	where $r_o$ is the area radius of $(\Sigma,\gamma)$, then $(\Sigma,\gamma,H_o,P_o,0)$ admits an extension $(M,g,K)$ that satisfies the DEC and contains no weakly trapped surfaces.
\end{introthm}
\begin{rem}
	The constants $\alpha$ and $\beta$ measure the roundness of $\gamma$ in some sense, with $\alpha\to 0$ and $\beta\to 1$ as $\gamma$ tends to a round metric in $C^{2,\tau}$ \cite{miao2018compact}. These constants, which are defined in Section \ref{subsec_collar}, are the same constants appearing in Theorem \ref{thm-mass-estimate} and Theorem \ref{thm-mass-est-2} below.
\end{rem}

Due to its flexibility, our construction further allows to formulate criteria for the data under which the extension satisfies a geometric condition, such as constant mean curvature (Corollary \ref{cor_cmc}), everywhere.

Inspired by the work of Lin \cite{Lin2024-Spin-Creased}, we further demonstrate that we can deform Bartnik data $(\Sigma,\gamma,H_o,P_o,0)$ to constant mean curvature (CMC) Bartnik data $(\Sigma,\widetilde{\gamma},\sqrt{H_o^2-P_o^2},0,0)$ (Theorem \ref{thm_main2}), reducing the construction of an admissible extension to the time symmetric case.

Both approaches allow us to estimate the Bartnik mass of the given data.

\begin{introthm}[Theorem \ref{thm-mb-bound}]\label{thm-mass-estimate}
	Let $(\Sigma, \gamma,H_o,P_o,0)$ be Bartnik data with $\gamma$ having strictly positive scalar curvature, where $H_o$, $P_o$ are constant with $H_o\ge \btr{P_o}>0$, satisfying
	\begin{equation*}
		\frac{H_o^2r_o^2}{4}<\frac{\beta}{1+\alpha}.
	\end{equation*}
	Then the Bartnik mass of $(\Sigma, \gamma,H_o,P_o,0)$ satisfies
	\begin{equation*}
		\m_B(\Sigma, \gamma,H_o,P_o,0)\leq \m_H(\S,\gamma,H_o,P_o,0)+\frac{H_or_o^2\sqrt{\alpha}(4-H_o^2r_o^2)}{8\sqrt{4\beta-(\alpha+1)H_o^2r_o^2}}.
	\end{equation*}
\end{introthm}
\begin{introthm}(Corollary \ref{cor-massbound2})\label{thm-mass-est-2}
	Let $(\S,\gamma,H_o,P_o,0)$ be Bartnik data with $\gamma$ having strictly positive scalar curvature, where $H_o$, $P_o$ are constant with $H_o\ge \btr{P_o}>0$, satisfying
	\begin{equation}
		R_\gamma>\frac32\mathcal H_o^2,
	\end{equation}
	where $\mathcal H_o=\sqrt{H_o^2-P_o^2}$.
	Then if 
	\begin{equation*}
		\frac{\mathcal H_o^2r_o^2}{4}<\frac{\beta}{1+\alpha},
	\end{equation*}
	then
	\begin{equation*}
		\m_B(\S,\gamma,H_o,P_o,0)\leq\left( 1+ \left(\frac{\alpha r_o^2\mathcal H_o^2}{4\beta-(1+\alpha)r_o^2\mathcal H_o^2}\right)^{1/2} \right)\m_H(\S,\gamma,H_o,P_o,0).
	\end{equation*}
\end{introthm}

This paper is structured as follows. Section \ref{sec_prelim} gives some standard definitions and sets the stage. Section \ref{sec_classIDS} introduces a class of initial data sets modeled on spherically symmetric graphs in Schwarzschild spacetimes, which we use to construct our Bartnik extensions. The bulk of the technical contributions are then given in Section \ref{sec_extension}, where the general construction of extensions is developed leading to Theorem \ref{thm-intromain}. Sections \ref{sec-estimate} and \ref{sec-reduct} then prove Theorems \ref{thm-mass-estimate} and \ref{thm-mass-est-2} respectively.

\section*{Acknowledgments}
SM is supported by Olle Engkvists stiftelse and foundations managed by The Royal Swedish Academy of Sciences. MW is supported by the Verg foundation.

\section{Preliminaries}\label{sec_prelim}
\subsection{Asymptotically flat initial data sets}
	An \emph{initial data set} for the Einstein equations is a triple $(M,g,K)$ given by a $3$-dimensional Riemannian manifold $(M,g)$ and a symmetric $(0,2)$-tensor $K$ satisfying the \emph{constraint equations},
	\begin{align} 
	R(g) + (\tr_{g} K)^2 - |K|^2_{g} &= 2 \mu,  \label{eq-CC-ham} \\
	\nabla^j(K_{ij} - ((\tr_{g} K) g_{ij}) )& =  J_i, \label{eq-CC-mom}
	\end{align}
	where $R(g)$ denotes the scalar curvature of $(M,g)$ and the source terms $\mu$ and $J$ represent the local \emph{energy density} and \emph{momentum density} of the matter fields, respectively\footnote{Compared to the common convention in physics, we have absorbed a factor of $8\pi$ into $\mu$ and $J$ for convenience.}. We say that $(M,g,K)$ satisfies the \emph{dominant energy condition} (DEC), which any reasonable physical matter model is expected to satisfy, if
	\begin{align}\label{eq_DEC}
	\mu\ge \btr{J}_g.
	\end{align}
	\begin{defn}\label{defi_asymflat}
		Let $(M,g,K)$ be an initial data set. We say that $(M,g,K)$ is \emph{asymptotically flat} if $\mu$, $J\in L^1(M)$ and there exists a compact set $\Omega\subset M$ such that $M\setminus \Omega$ is diffeomorphic to $\R^3 \setminus B_R(0)$ for some $R>0$, and with respect to these coordinates $g$ and $K$ satisfy the decay conditions
		\begin{align*}
		g_{ij} &=\delta_{ij} + o_2(|x|^{-\frac{1}{2}}), \\
		K_{ij} &=o_1(|x|^{-\frac32}).
		\end{align*}
	\end{defn}
	The total energy and momentum of an asymptotically flat manifold -- the \emph{ADM energy} and \emph{ADM momentum} -- are defined respectively by
	\begin{align*}
	E &= \frac{1}{16\pi} \lim\limits_{r\to\infty} \int_{S_r} (g_{ij,i} - g_{ii,j})\nu^j \, dA_{S_{r}}, \\
	p_i &= \frac{1}{8\pi} \lim\limits_{r\to\infty}  \int_{S_r} (K_{ij} -  (\tr_g K)g_{ij})\nu^j \, dA_{S_{r}},
	\end{align*}
	where $\nu$ is the outward unit normal to round spheres $S_r$ in the asymptotic end. If the DEC is satisfied then $E^2\geq p_i p^i$ by the positive mass theorem \cite{schoen1981proof,witten1981new} and the \emph{ADM mass} is defined as
	\[
	m=\sqrt{E^2-|p|^2}.
	\]
	Let $\Sigma$ be an orientable hypersurface in $(M,g,K)$ and denote by $H$ its mean curvature with respect to a unit normal $\nu$. The null expansions $\theta_\pm$ of $\Sigma$ in $(M,g,K)$ are then defined as
	\begin{align*}
	\theta_\pm=H\pm P,
	\end{align*}
	where $P= \tr_{\Sigma}K=\tr_MK-K(\nu,\nu)$. If $\theta_\pm=0$ along $\Sigma$, we call $\Sigma$ a \emph{marginally outer / inner trapped surface} (MOTS/MITS) in $(M,g,K)$. We further define the spacetime mean curvature of $\Sigma$ as $\mathcal{H}^2=H^2-P^2$. If $\mathcal{H}^2=0$ we call $\Sigma$ a \emph{generalised apparent horizon}. If $(M,g)$ embeds isometrically into a spacetime $(\mathbf{M},\mathbf{g})$ with timelike, future-pointing unit normal $\vec{n}$ and second fundamental form $K$, then the codimension-$2$ mean curvature vector $\vec{\mathcal{H}}$ of any orientable surface $\Sigma\subseteq M\subseteq \mathbf{M}$ can be decomposed as
	\[
	\vec{\mathcal{H}}=-H\nu+P\vec{n}.
	\]
	In particular, the length of the mean curvature vector is given by
	\begin{align*}
	\mathbf{g}(\vec{\mathcal{H}},\vec{\mathcal{H}})=\theta_+\theta_-=\mathcal{H}^2.
	\end{align*}
	From this, we see that a MOTS/MITS is always a generalised apparent horizon. However, the converse statement is not true in general (see \cite{carmars} for a counterexample).
	
	The \emph{Hawking mass} of a surface $\Sigma$ is given by
	\begin{equation}\label{eq-Hawkingmassdefn}
		\mathfrak{m}_{{H}}(\Sigma)=\sqrt{\frac{\btr{\Sigma}}{16\pi}}\left(1-\frac{1}{16\pi}\int_\Sigma \mathcal{H}^2\d\mu_\Sigma\right),
	\end{equation}
	where $\d\mu_\Sigma$ denotes the induced volume form and $\btr{\Sigma}$ the area of $\Sigma$, respectively.

\subsection{The Bartnik mass}\label{Ssec-Bartnik}
The Hawking mass defined above is one example of a \emph{quasi-local mass}, which assigns a scalar to a closed hypersurface\footnote{Sometimes a quasi-local mass is defined in terms of the domain bounded by the closed hypersurface instead.} $\Sigma$ in an initial data set, representing the total energy contained within $\Sigma$. The Hawking mass is considered by many to be a good definition of quasi-local mass in the case that $\Sigma$ is round, but it is well-known that the Hawking mass under-estimates the physical notion of mass. For example, it is negative for all non-round spheres in $\mathbb R^3$ (viewed as an initial data set with $K\equiv0$).

The Bartnik mass is another definition of quasi-local mass, which is generally believed to better describe the physical notion of mass but suffers the problem of being essentially impossible to compute in general. We briefly recall the definition.

We define \textit{Bartnik data} to be a tuple $(\Sigma,\gamma,H,P,\omega^\perp)$ consisting of a closed $2$-dimensional manifold $(\Sigma,\gamma)$ equipped with functions $H$ and $P$ and a one-form $\omega^\perp$. Usually (and indeed here) $\Sigma$ is taken to be a sphere. The Bartnik data is to be understood as being induced from an initial data set $(M,g,K)$ where $H$ is the mean curvature of $\Sigma\subset M$, $P=\trace_\Sigma(K)$ and $\omega^\perp_A=K_{iA}\nu^i$ where $A$ denotes an index on $\Sigma$.

The \emph{Bartnik mass} \cite{Bartnik1989-QLM-PRL} of such Bartnik data is given by
\begin{equation*}
	\m_B(\S,\gamma,H,P,\omega^\perp)=\inf\limits_{(M,g,K)\in\mathcal A(\S,\gamma,H,P,\omega^\perp)}\left\{ \m_{ADM}(M,g,K)\right\},
\end{equation*}
where $\mathcal A(\S,\gamma,H,P,\omega^\perp)$ denotes the set of \emph{admissible extensions} of the Bartnik data $(\S,\gamma,H,P,\omega^\perp)$. Namely, the set of asymptotically flat initial data sets $(M,g,K)$ with boundary $\partial M\cong\Sigma$ satisfying:
\begin{itemize}
	\item the induced Bartnik data on $\partial M$ agrees with $(\S,\gamma,H,P,\omega^\perp)$ with respect to the inward-pointing unit normal,
	\item $(M,g,K)$ satisfies the DEC,
	\item $(M,g,K)$ contains no generalised apparent horizons enclosing $\partial M$.
\end{itemize}
That is, the Bartnik mass is a localisation of the ADM mass and the admissible extensions are those that do not hide $\Sigma$ behind a horizon, while ensuring the positive mass theorem applies by imposing the DEC, both on $M$ and in a distributional sense where the boundary would be glued to a compact manifold with boundary inducing the Bartnik data. See \cite{McCormick2024-Overview-Bartnik} for a detailed discussion on the definition of the Bartnik mass and admissible extensions.

Although there is no known way to calculate this infimum in general, the construction of admissible extensions with known ADM mass immediately provides an upper bound for the Bartnik mass. Indeed many such extensions have been constructed in the time-symmetric setting in recent years (for example, \cite{CabreraPacheco2017-AF-CMC-Bartnik,MiaoPiubello2024,MiaoWangXie2020}). To the best knowledge of the authors, the only known admissible extensions constructed for the Bartnik mass outside of time-symmetry is due to Lin in his PhD thesis \cite{Lin2024-Spin-Creased} where admissible extensions of Bartnik data corresponding to a MOTS is constructed, showing the Bartnik mass of such Bartnik data is bounded above by the Hawking mass, which implies the Bartnik mass of a MOTS is equal to the Hawking mass of a MOTS if the Penrose inequality holds. This is directly analogous to the work of Mantoulidis and Schoen \cite{MantoulidisSchoen2015-Apparent-Horizons} in the time-symmetric setting.

In this article we construct new admissible extensions outside of time-symmetry and estimate the Bartnik mass from above.
	
\section{A class of almost spherically symmetric initial data sets}\label{sec_classIDS}

In this section, we describe the class of initial data sets that we want to use for our collar construction in Section \ref{sec_extension}. We note that a similar class of initial data sets has been considered by Cabrera Pacheco and the second-named author in \cite{CabreraPachecoWolff2024-NonTimeSym-Penrose}. In fact, all the initial data sets considered here are a special case of the class in \cite{CabreraPachecoWolff2024-NonTimeSym-Penrose}. Here, our choice of initial data sets is closely modeled on spherically symmetric graphs in the Schwarzschild spacetime.

\subsection{Spherically symmetric Schwarzschild graphs}\label{subsec_schwarzschild}

Consider a spherically symmetric Lorentzian manifold $(V,\mathfrak g)$ given by $V=\R\times I\times\Sbb^2$,
\begin{align}\label{eq_spacetimeclass}
\mathfrak{g}=-h(r)\d t^2+\frac{1}{h(r)}\d r^2+r^2\gamma_*,
\end{align}
where $h$ is a smooth, positive function on some interval $I\subseteq (0,\infty)$, and $\gamma_*$ denotes the round metric on the $2$-sphere. Note that Schwarzschild spacetimes lie in this class with $h(r)=1-\frac{2m}{r}$. We will consider spacelike hypersurfaces of such spacetimes given by spherically symmetric graphs over the $t=0$ slice, defined by $\{t=T(r)\}$ for some graph function $T$. As shown by the second-named author in \cite{Wolff2024-NEC-Umbilic}, the induced metric and second fundamental form of any spherically symmetric graph satisfy
\begin{align*}
	g^T&=\frac{1}{h_T(r)}\d r^2+r^2\gamma_*,\\
	K^T&=a(r)\d r^2+b(r)r^2\gamma_*,
\end{align*}
where $h_T$, $a$, $b$ are smooth functions that satisfy
\begin{align*}
	\frac{h_T-h}{r^2}&=b^2,\\
	\frac{(h_T-h)'}{r}&=2h_Tab.
\end{align*}
In fact, the discussion on spherically symmetric graphs in \cite{Wolff2024-NEC-Umbilic} shows that $(M_T,g^T,K^T)$ is fully determined by the smooth function $x=rb$, and it holds that
\begin{align*}
	g^T&=\frac{1}{h+x^2}\d r^2+r^2\gamma_*,\\
	K^T&=\frac{x'}{h+x^2}\d r^2+xr\gamma_*.
\end{align*}
Choosing a suitable constant $c$ such that\footnote{In the Schwarzschild case, as $x^2\ge 0$, it is always possible to choose $c\ge \max(2m,0)$ in view of the coordinate transformation in time symmetry. See, for example, \cite{MantoulidisSchoen2015-Apparent-Horizons}.}
\[
s(r)=\int_{c}^r\frac{1}{\sqrt{h_T(\rho)}}\d \rho,
\]
is well-defined, and therefore smooth and strictly increasing, a change of coordinates yields
\begin{align*}
	g^T&=\d s^2+v^2(s)\gamma_*,\\
	K^T&= x'(v(s))\d s^2+x(v(s))v(s)\gamma_*,
\end{align*}
where $v=s^{-1}$ satisfies
\[
v'(s)=\sqrt{h(v(s))+x^2(v(s))}.
\]
Assuming vacuum, the constraint map \eqref{eq-CC-ham}, \eqref{eq-CC-mom} vanishes identically. More generally, one can check that
\begin{align*}
	\mu_T&=\frac{2}{v^2}\left(1-h(v)+vh'\right),\\
	J_T&=0,
\end{align*}
for a spherically symmetric graph in a spacetime $(V,\mathfrak{g})$ of the form \eqref{eq_spacetimeclass}.
Before considering a more general class of initial data sets that will retain some key properties, we shall give some examples of relevant choices for $x$ (in the Schwarzschild case).
\begin{rem}\label{rem_examples}\,
\begin{enumerate}
	\item[(i)] Similar to \cite{Wolff2024-NEC-Umbilic}, an ODE argument shows that any spherically symmetric graph with constant mean curvature $3C_1$ corresponds to 
	\[
	x=C_1r-\frac{C_2}{r^2}
	\]
	for a constant $C_2\in\R$. $C_1=0$ corresponds to a maximal initial data set with $\tr_gK=0$, while $C_2=0$ corresponds to an umbilical initial data set with $K=C_1g$. The umbilical case further corresponds to the isometric embedding of the spatial Schwarzschild anti-de\,Sitter manifold for a suitable cosmological constant.
	\item[(ii)] $x=C_3\sqrt{\frac{2}{r}}$ for some constant $C_3\in\mathbb R$ corresponds to $g$ being isometric to the spatial Schwarzschild manifold of mass $m-C_3^2$. In particular, if $m\ge 0$, $x=\sqrt{\frac{2m}{r}}$ corresponds to an embedding of Euclidean space into the Schwarzschild spacetime. Note that such graphs do not give asymptotically flat initial data because $K$ decays too slowly\footnote{In fact, $K$ has exactly borderline decay, $K=O(r^{-\frac{3}{2}})$.}.
\end{enumerate}
\end{rem}

\subsection{A model class of initial data sets}\label{subsec_ClassIDS}

Similar to \cite{CabreraPachecoWolff2024-NonTimeSym-Penrose} we consider a slight generalization of spherically symmetric initial data sets that will retain a large amount of the structural properties due to their block diagonal nature. In fact, the initial data sets considered in this section are a special case of the class of initial data sets considered in \cite{CabreraPachecoWolff2024-NonTimeSym-Penrose}.

Let $f\colon(a,b)\to(0,\infty)$ be a smooth, positive function, and let $\{\gamma(s)\}_{s\in(a,b)}$ be a smooth path of metrics on $\mathbb{S}^2$ such that $\operatorname{tr}_{\gamma(s)}\gamma'(s)=0$.

We consider initial data sets $(M,g,K)$ of the form $M=(a,b)\times\mathbb{S}^2$, 
\begin{align*}
	g&=A^2\d s^2+f(s)^2\gamma(s),\\
	K&=A^2x'(f(s))\d s^2+x(f(s))f(s)\gamma(s).
\end{align*}

Using coordinates $s,x^I$ on $M$, where $x^I$ denote a choice of coordinates on $\mathbb{S}^2$, it is straightforward to check that all non-trivial Christoffel symbols are given by
\begin{align*}
	\Gamma_{sK}^I&=\frac{f'}{f}\delta_K^I+\frac{1}{2}\gamma^{IJ}\gamma'_{JK},\\
	\Gamma_{KL}^s&=-\frac{f\,f'}{A^2}\gamma_{KL}-\frac{1}{2}\frac{f^2}{A^2}\gamma'_{KL},\\
	\Gamma_{KL}^I&={}^s\Gamma_{KL}^I,
\end{align*}
where ${}^s\Gamma_{KL}^I$ denotes the Christoffel symbols with respect to the metric $\gamma(s)$, where here and in the following we often omit the explicit $s$-dependency of the path of metrics $\gamma$ for convenience.

We first observe that the momentum constraint always vanishes.
\begin{lem}\label{lem_J}
	Let $(M,g,K)$ be as above. Then $J\equiv 0$.
\end{lem}
\begin{proof}
	Recall that 
	\[
	J=\operatorname{div}(K-\trace Kg)=\dive K-\d \trace K.
	\]
	As $\trace K=x'(f)+2\frac{x(f)}{f}$, we find
	\[
	\d \trace K=f'\left(x''(f)+\frac{2x'(f)}{f}-\frac{2x(f)}{f^2}\right)\d s.
	\]
	Using the above Christoffel symbols and recalling that $\trace_\gamma \gamma'=0$, it is straightforward to check that
	\[
	\dive K_I=\frac{x(f)}{f}\dive_{\gamma}\gamma=0,
	\]
	and 
	\[
	\dive K_s=x''(f) f'+2x'(f)\frac{f'}{f}-2x(f)\frac{f'}{f^2}.
	\]
	The claim directly follows.
\end{proof}
\begin{lem}\label{lem_mu}
	Let $(M,g,K)$ be as above. Then
	\[
	\mu=\frac{\operatorname{R}_s}{2f^2}-\frac{1}{A^2f^2}\left(f'^2+2ff''\right)+\frac{1}{f^2}\left(x(f)^2+2x(f)x'(f)f\right)-\frac{1}{8A^2}\vert\gamma'\vert^2,
	\]
	where $\operatorname{R}_s$ denotes the scalar curvature of $\gamma(s)$.
\end{lem}
\begin{proof}
	Recall the scalar curvature of a metric as above is given by (see for example, \cite{miao2018compact})
	\[
	\operatorname{R}=\frac{\operatorname{R}_s}{f^2}-\frac{2}{A^2f^2}\left(f'^2+2ff''\right)-\frac{1}{4A^2}\vert\gamma'\vert^2.
	\]
	
	\noindent Using that
	\begin{align*}
		\trace K&=x'(f)+\frac{2x(f)}{f},\\
		\vert K\vert ^2&=(x'(f))^2+\frac{2x(f)^2}{f^2},
	\end{align*}
	the claim follows directly.
\end{proof}
\begin{cor}\label{cor_DEC}
	Let $(M,g,K)$ be as above. Then $(M,g,K)$ satisfies the dominant energy condition if and only if
	\[
	\frac{\operatorname{R}_s}{2f^2}-\frac{1}{A^2f^2}\left(f'^2+2ff''\right)+\frac{1}{f^2}\left(x^2+2xx'f\right)-\frac{1}{8A^2}\vert\gamma'\vert^2\ge 0,
	\]
	and it satisfies it in the strict sense if and only if the above inequality is strict.
\end{cor}

\section{Non-time-symmetric extensions of Bartnik data}\label{sec_extension}

In this section, we consider Bartnik data $(\Sigma,\gamma_o, H_o,P_o,0)$, where $H_o$ and $P_o$ are constant, and find criteria to construct a non-time-symmetric initial data set much in the spirit of \cite{CabreraPacheco2017-AF-CMC-Bartnik}. In particular, we will first construct a collar extension that satisfies the dominant energy condition in the strict sense and deforms the initial metric to a round one. We then use a modified gluing and bending lemma to attach the collar to a spherically symmetric graph in a Schwarzschild spacetime of suitable mass.

We will discuss the applications of this to estimating the Bartnik mass of the given data in the non-time-symmetric setting in more detail in Section \ref{sec-estimate}. This follows the same idea as the time-symmetric case \cite{CabreraPacheco2017-AF-CMC-Bartnik,MantoulidisSchoen2015-Apparent-Horizons}.

\subsection{Bartnik data}\label{subsec_extension_setup}

We consider Bartnik data $(\Sigma, \gamma_o,H_o,P_o,0)$, where $H_o$ and $P_o$ are constant, and we in particular restrict ourselves to the case $\omega^\perp\equiv0$. We further assume that the data is untrapped, which is to say that $\mathcal{H}^2_o=H_o^2-P_o^2\ge 0$ {}\footnote{In fact, we assume that $H_o\ge \btr{P_o}>0$.}, and has non-negative Hawking mass
\[
	\m_o=\m_{{H}}(\Sigma)=\sqrt{\frac{\btr{\Sigma}}{16\pi}}\left(1-\frac{1}{16\pi}\int_{\Sigma}\mathcal{H}_o^2\d\mu_\Sigma\right)\ge 0,
\]
which are both physically reasonable assumptions. We further denote the area radius of $\Sigma$ by $r_o$, where $r_o$ is determined by $\btr{\Sigma}_{\gamma_o}=4\pi r_o^2$, noting that the Hawking mass satisfies $2\m_o\leq r_o$ by the untrapped condition.

Although we avoid making a specific choice of $x(r)$ in this section, we will always assume that it is a smooth function on $[2\m_o,\infty)$ such that $x(r_o)\not=0$. 

\subsection{Collar construction}\label{subsec_collar}

Our first aim is to construct a collar that deforms the metric $\gamma_o$ to a round metric and satisfies the DEC in a strict sense.

To achieve our first goal, we choose a path of metrics $\gamma(s)$ where $s\in[0,1]$, such that
\begin{itemize}
	\item[(i)] $\gamma(0)=\gamma_o$ and $\gamma(1)$ is round,
	\item[(ii)] $\trace_{\gamma(s)}\gamma'(s)=0$,
	\item[(iii)] $\gamma(s)=\gamma(1)$ on $[1-\varepsilon,1]$ for some $\varepsilon>0$.
\end{itemize}
Note that such a path exists for any initial metric $\gamma_o$ on $\mathbb{S}^2$. For example, by the Uniformisation Theorem (eg. \cite{MantoulidisSchoen2015-Apparent-Horizons}) or a suitable time reparametrisation of Hamilton's Ricci flow combined with a family of diffeomorphism on $\mathbb{S}^2$ to ensure property (ii). Note that, up to a change of coordinates, this means we have
\[
\d\mu_o=\d\mu_s=r_o^2\d{\mu_*},
\]
where $\d{\mu_*}$ denotes the volume element of the standard round metric.

Further, consider the solution $v$ of the ODE
\begin{align}\label{eq_ODE_v}
	v'(s)=\sqrt{1-\frac{2\overline{m}}{v(s)}+x^2(v(s))},
\end{align}
as before, where we consider the parameter $\overline{m}$ and the smooth function $x$ as free data. For convenience, we now consider a translation $u(s)=v(s+T)$ such that $u(0)=r_o$. We note that $u$ still satisfies \eqref{eq_ODE_v} and that such a translation is always possible if $r_o\ge 2\overline{m}$ as $v(T)=r_o$ is equivalent to
\[
T=\int_{v(0)}^{r_o}\frac{1}{\sqrt{1-\frac{2\overline{m}}{\rho}+x^2(\rho)}}\d\rho.
\]

We now make an ansatz for the collar extension initial data $(g,K)$ on $M=[0,1]\times \mathbb{S}^2$,
\begin{align}
	\begin{split}\label{eq_ansatz_collar}
		g&=A^2\d s^2+u(Aks)^2r_o^{-2}\gamma(s),\\
		K&=D\left[A^2x'(u(Aks))\d s^2+x(u(Aks))u(Aks)r_o^{-2}\gamma(s)\right],
	\end{split}
\end{align}
for constants $A, D, k$. Writing simply $x$ and $u$ in the following for simplicity, it is straightforward to check from Corollary \ref{cor_DEC}, using $f=u(Aks)$, $\widetilde{x}(r)=Dx(r)$, and $\widetilde{\gamma}(s)=r_o^{-2}\gamma(s)$, that the collar extension satisfies the strict DEC if and only if
\[
0<\frac{r_o^2\operatorname{R}_s}{u^2}-\frac{2k^2}{u^2}-\frac{1}{4A^2}\vert\gamma'\vert^2-\frac{2(k^2-D^2)}{u^2}\left(x^2+2xx'u\right),
\]
where we additionally used that $u$ satisfies \eqref{eq_ODE_v}. Since the sign of the last term in the parentheses will play a role here, it will be convenient to define the function 
\begin{align}
	G_x(r)&=x^2(r)+2x(r)x'(r)r.\label{eq_G1}
\end{align}

\begin{lem}\label{lem_stmc}
	Let $\Sigma_s$ denote the level sets of $s$. Then
	\begin{align*}
		H(s)&=\frac{2k}{u}\sqrt{1-\frac{2\overline{m}}{u}+x^2},\\
		P(s)&=\frac{2Dx}{u}.
	\end{align*}
	In particular,
	\[	
	\mathcal{H}^2(s)=\frac{4k^2}{u^2}\left(1-\frac{2\overline{m}}{u}\right)+\frac{4x^2}{u^2}\left(k^2-D^2\right).
	\]
\end{lem}

\begin{proof}
	From \eqref{eq_ansatz_collar}, it is straightforward to see that $\nu_s=\frac{1}{A}\partial_s$ and that the second fundamental form $A(s)$ of $\Sigma_s$ thus satisfies
	\[
	A(s)_{IJ}=-g(\nabla_{\partial_I}\partial_J,\nu_s)=-A\Gamma_{IJ}^s.
	\]
	Hence, 
	\[
	A(s)_{IJ}=k\frac{u u'}{r_o^2}\gamma_{IJ}+\frac{1}{2}\frac{u^2}{A}\gamma'_{IJ}.
	\]
	The assertion for $H(s)$ then follows upon taking a trace, recalling that $u$ satisfies \eqref{eq_ODE_v} and $\trace_\gamma \gamma'=0$. As $P(s)=\trace_{\Sigma_s}K$, the claim immediately follows from observing that $\Sigma_s$ has induced metric $u^2r_o^{-2}\gamma$. Finally, the last claim holds by definition as
	\[
	\mathcal{H}^2(s)=H(s)^2-P(s)^2. 
	\]
\end{proof}
We now define the constants
\begin{align*}
	\alpha&=\max\limits_{[0,1]\times \Sigma}\frac{1}{4}\vert\gamma'\vert^2\ge 0,\\
	\beta&=\frac{1}{2}\min\limits_{[0,1]\times \Sigma}(r_o^2R(\gamma(s)))\le 1,
\end{align*}
which only depend on the path of metrics $\gamma(s)$, or once the method of path construction is fixed, only on $\gamma_o$. We note that the upper bound on $\beta$ follows from the Gauss--Bonnet Theorem, and in fact $\alpha=0$ if and only if $\beta=1$, in which case $\gamma_o$ is a round metric. These constants were introduced by Miao and Xie \cite{miao2018compact} and used in several of the estimates for the Bartnik mass in the time-symmetric setting (for example, \cite{CabreraPacheco2017-AF-CMC-Bartnik,CabreraPacheco2018-AH-Bartnik,MiaoPiubello2024,MiaoWangXie2020}). As shown by Miao and Xie \cite{miao2018compact}, it is possible to choose the path $\gamma(s)$ such that if $\gamma$ is $C^{2,\tau}$-close to a round metric $\gamma_*$, these constants satisfy
\begin{equation*}
	\alpha\leq C \| \gamma - \gamma_*\|_{C^{0,\tau}(\Sigma)}^2,\qquad \text{ and }(1-\beta)\leq C\|\gamma-\gamma_*\|_{C^{2,\tau}(\Sigma)}.
\end{equation*}

In the following, we will always assume that $\beta>0$. Note that positive scalar curvature is preserved under Hamilton's Ricci flow, so provided that $\gamma_o$ has positive scalar curvature we can always ensure $\beta>0$ by constructing the path of metrics this way.

Using Lemma \ref{lem_stmc}, we conclude that the collar extension satisfies the DEC in the strict sense in terms of the Bartnik data $\gamma_o,H_o,P_o$ if
\begin{align}\label{eq_dec_collar}
	0<2\beta-2k^2-\frac{u^2\alpha}{A^2}+\frac{G_x(u)}{x(r_o)^2}\left(2k^2\left(1-\frac{2\overline{m}}{r_o}\right)-\frac{r_o^2}{2}\mathcal{H}^2(0)\right),
\end{align} 
where we recall that $G_x(r)=x^2(r)+2x(r)x'(r)r$.
Note further that the constants $D$, $k$ are fully determined by the Bartnik data and a choice of $\overline m$ and $x$, since Lemma \ref{lem_stmc} gives 
\begin{align}\begin{split}\label{eq-kHP}
	H_o&=H(0)=\frac{2k}{r_o}\sqrt{1-\frac{2\overline{m}}{r_o}+x^2(r_o)},\\
	P_o&=P(0)=2D\frac{x(r_o)}{r_o}.\end{split}
\end{align}

We observe that the first three terms in \eqref{eq_dec_collar} are exactly the same terms as in the time-symmetric case (see, for example, \cite{CabreraPacheco2017-AF-CMC-Bartnik}), so it suffices to ensure that the last term remains non-negative, and treat the remaining similar to the time-symmetric case.

Note that
\begin{align*}
	2k^2\left(1-\frac{2\overline{m}}{r_o}\right)-\frac{r_o^2}{2}\mathcal{H}^2(0)
	&=\frac{r_o^2}{2}\left(P_o^2-\frac{x^2(r_o)}{1-\frac{2\overline{m}}{r_o}+x^2(r_o)}H_o^2\right).
\end{align*}
Recall that we always assume that $\mathcal{H}^2_o=H_o^2-P_o^2\ge 0$. Hence, if $2\overline{m}\le {r_o}$, we observe that the above term is non-positive for $r_o=2\overline{m}$, and non-negative for MOTS/MITS data. It vanishes identically iff
\begin{align}\label{eq_k^2eqD^2}
x^2(r_o)H_o^2=\left(1-\frac{2\overline{m}}{r_o}+x^2(r_o)\right)P_o^2,
\end{align}
which we will later ensure by choosing the parameter $\overline{m}$ appropriately. For example, in the case $r_o=2\overline{m}$, \eqref{eq_k^2eqD^2} is satisfied for MOTS/MITS data. Note that, unless this term vanishes identically, the sign of the last term in \eqref{eq_dec_collar} still depends on the sign of $G_x$. In what follows, we develop the collar construction in a general way as it may be of independent interest, however later we will restrict our attention to construct extensions such that \eqref{eq_k^2eqD^2} is satisfied in the collar region.

We are now ready to state our collar construction.

\begin{prop}\label{prop_collarextension}\,\newline
	Let $(\mathbb{S}^2, \gamma_o,H_o,P_o)$ with $H_o\ge \vert P_o\vert >0$, area radius $r_o$ and $\alpha$, $\beta$ as above with $\alpha\ge 0$, $0<\beta\le 1$. Let $x\colon[r_o,\infty)\to\mathbb{R}$ be a smooth function with
	$x(r_o)\not=0$ and 
	\begin{align*}
	\vert x(r)\vert \le L_1+L_2r
	\end{align*}
	for non-negative constants $L_1$, $L_2$. Let $\overline{m}\in(-\infty,\frac{1}{2}r_o]$. Assume that
	\begin{align*}
	(a)\begin{cases}
			H_o^2\ge \frac{1-\frac{2\overline{m}}{r_o}+x^2(r_o)}{x^2(r_o)}P_o^2\text{ and } G_x\le 0\\
			or\\
			H_o^2\le \frac{1-\frac{2\overline{m}}{r_o}+x^2(r_o)}{x^2(r_o)}P_o^2\text{ and } G_x\ge 0.
	\end{cases}
	\end{align*}
	Assume further that
	\[
		H_o^2< C\frac{4}{r_o^2}\left(1-\frac{2\overline{m}}{r_o}+x^2(r_o)\right),
	\]
	where $C=C(\alpha,\beta,r_o,m,L_1,L_2)\le 1$ is defined as
	\begin{align*}
	C=\begin{cases} \frac{\beta}{1+{\alpha}(1-\frac{2\overline{m}}{r_o}+L_1^2)},&\text{if }L_2=0,\\
	\frac{\beta}{1+\alpha\left(e^{2L_2}r_o^2+\frac{(e^{L_2}-1)^2}{L_2^2}\left(1-\frac{2\overline{m}}{r_o}+L_1^2\right)\right)},&\text{if }L_2>0.
	\end{cases}
	\end{align*}
	Then $(\Sigma,\gamma_o,H_o,P_o)$ admits a smooth collar extension $(M,g,K)$ with $M=[0,1]\times \mathbb{S}^2$,
	\begin{align*}
	g&=A^2\d s^2+u(Aks)^2r_o^{-2}\gamma(s),\\
	K&=D\left(A^2x'(u(Aks))\d s^2+x(u(Aks))u(Aks)r_o^{-2}\gamma(s)\right),
	\end{align*}
	that strictly satisfies the dominant energy condition, where $u$ satisfies $u(0)=r_o$ and solves
	\[
	u'(s)=\sqrt{1-\frac{2\overline{m}}{u}+x^2(u)},
	\]	
	and $D=D(x(r_o),P_o)$, $A=A(\alpha,\beta,r_o,\overline{m},L_1,L_2,H_o)$ are fully determined by the Bartnik data and $\overline{m}$.
\end{prop}
\begin{rem}\label{rem_collarextension}\,
	Note that the additional condition on $x$ 
	\[ 	
	\btr{x(r)}\le L_1+L_2r
	\] 
	is necessary to obtain some growth control on $u$. We expect that there is a plethora of possible conditions on $x$ to achieve this. Here, we chose this rather simple condition as it is satisfied for all examples in Remark \ref{rem_examples} for suitable choices of $L_1$ and $L_2$.
\end{rem}

\begin{proof}
	Using condition (a), it suffices to prove that 
	\[
		2\beta-2k^2-\frac{u^2\alpha}{A^2}> 0,
	\]
	where we make use of \eqref{eq_dec_collar}. Recall that 
	\[
		H_o^2=\frac{4k^2}{r_o^2}\left(1-\frac{2\overline{m}}{r_o}+x^2(r_o)\right)
	\]
	so $k^2< C=C(\alpha,\beta,L_1,L_2)\le \beta$ and therefore if $\alpha=0$ there is nothing left to prove. We now assume $\alpha>0$ {}\footnote{As $\btr{\gamma'}(s)=0<\alpha$ for $s\in (1-\varepsilon,1]$ in this case, we only need that the above inequality holds strictly on $[0,1-\varepsilon]$.} and consider two cases, making the distinction between $L_2=0$ and $L_2>0$.
	
	We first assume that $L_2=0$. Note that $u'\le \sqrt{1-\frac{2\overline{m}}{r_o}+L_1^2}$ and therefore
	\[
	u^2(Aks)\le 2 \left(1-\frac{2\overline{m}}{r_o}+L_1^2\right)A^2k^2s^2+2r_o^2.
	\]
	Hence, we find that by defining
	\[
	A={r_o}\left(\frac{\alpha}{\beta-k^2\left(1+{\alpha}\left(1-\frac{2\overline{m}}{r_o}+L_1^2\right)\right)}\right)^{\frac{1}{2}}>0,
	\]
	which is well defined as $k^2<C$, we have
	\begin{align*}
	2\beta-2k^2-\frac{u^2\alpha}{A^2}
	&\ge 2\beta-2k^2-2\left(1-\frac{2\overline{m}}{r_o}+L_1^2\right)\alpha k^2s^2-2\frac{r_o^2\alpha}{A^2}\\
	&\ge 2\beta-2k^2-2\left(1-\frac{2\overline{m}}{r_o}+L_1^2\right)\alpha k^2-2\frac{r_o^2\alpha}{A^2}\\
	&=0,
	\end{align*}
	where we note that the second inquality is strict unless $s=1$. This shows the claim for $L_2=0$.
	
	We now assume $L_2>0$. Then $u'\le \sqrt{1-\frac{2\overline{m}}{r_o}+L_1^2}+L_2u$, therefore ODE comparison yields
	\[
		u(Aks)\le \frac{\sqrt{1-\frac{2\overline{m}}{r_o}+L_1^2}+L_2r_o}{L_2}e^{L_2Aks}-\frac{\sqrt{1-\frac{2\overline{m}}{r_o}+L_1^2}}{L_2},
	\] 
	so
	\[
		u^2(Aks)\le 2e^{2L_2}r_o^2+2\frac{(e^{L_2}-1)^2}{L_2^2}\left(1-\frac{2\overline{m}}{r_o}+L_1^2\right).
	\]
	We now define 
	\[
		A=\left(\frac{\alpha\left(e^{2L_2}r_o^2+\frac{(e^{L_2}-1)^2}{L_2^2}\left(1-\frac{2\overline{m}}{r_o}+L_1^2\right)\right)}{\beta-k^2}\right)^\frac{1}{2}
	\]
	which is well-defined as $k^2<C\le \beta$. Note moreover that $k^2<C$ implies that $Ak\le 1$. Similar to before, we conclude that this choice of $A$ and the above ODE comparison yield
	\begin{align*}
		2\beta-2k^2-\frac{u^2\alpha}{A^2}
		&=
		\frac{1}{A^2}\left((2\beta-2k^2)A^2-\alpha u^2\right)
		\ge 0
	\end{align*}
	where the inequality is strict except possibly at $s=1$. This shows the claim for $L_2>0$.
\end{proof}

While Proposition \ref{prop_collarextension} should prove to be a flexible tool, the main aim of this paper is to give some clean conditions on the existence of Bartnik extensions and in particular to give an upper bound on the Bartnik mass. To this end, we will from now on always ensure that condition (a) in Proposition \ref{prop_collarextension} is satisfied. We can do this by either restricting the choice of $x$ so that $G_x\equiv0$ or by ensuring the equality case of \eqref{eq_k^2eqD^2} so that condition (a) is satisfied irrespective of the sign of $G_x$\footnote{In particular, the last term in the energy density $\mu$ vanishes identically. In fact, this exact term also appears in the derivative of the Hawking mass. That is, $G_x\equiv 0$ or condition \eqref{eq_k^2eqD^2} later additionally serve to give a sharper bound on the Bartnik mass.}. In this section we impose the latter by choosing $\overline{m}$ carefully, and then in Section \ref{sec-estimate} we instead prescribe the function $x$.

Rearranging \eqref{eq_k^2eqD^2} gives
\begin{align}\label{eq_massparameter}
	\overline{m}=\frac{r_o}{2}\left(1-(H_o^2-P_o^2)\frac{x(r_o)^2}{P_o^2}\right)\in (-\infty,\frac{r_o}{2}],
\end{align}
which we from now on take as the definition of the choice of parameter $\overline{m}$. Note that for this choice of $\overline{m}$, 
\[
	H_o^2\le C\frac{4}{r_o^2}\left(1-\frac{2\overline{m}}{r_o}+x(r_o)^2\right)\Leftrightarrow P_o^2\le C\frac{4}{r_o^2}x(r_o)^2,
\]
which as $C<1$ also shows that $\overline{m}<\m_o$ unless $\mathcal{H}^2_o=0$, in which case $\overline{m}=\m_o=\frac{r_o}{2}$. Finally, note that the constant $C$ depends itself on
\[
	1-\frac{2\overline{m}}{r_o}=\mathcal{H}_o^2\frac{x(r_o)^2}{P_o^2}.
\]
By a further rearrangement, we obtain the following simplified collar construction.

\begin{cor}[Simplified Collar Construction]\label{prop_collarextension_simplified}\,\newline
	Let $(\mathbb{S}^2, \gamma_o,H_o,P_o)$ with $H_o\ge \vert P_o\vert >0$, area radius $r_o$ and $\alpha$, $\beta$ as above with $\alpha\ge 0$, $0<\beta\le 1$. Let $x\colon[r_o,\infty)\to\mathbb{R}$ be a smooth function with
	$x(r_o)\not=0$ and $\vert x(r)\vert \le L_1+L_2r$ for non-negative constants $L_1$, $L_2$. Furthermore let
	\begin{align*}
		C_1&=
		\begin{cases}
			\frac{\beta}{\alpha}&L_2=0,\\
			\frac{\beta L_2^2}{\alpha(e^{L_2}-1)^2}&L_2>0,
		\end{cases}\\
		C_2&=
		\begin{cases}
			\frac{4\beta-\alpha r_o^2\mathcal{H}_o^2}{r_o^2\left(1+\alpha L_1^2\right)}&L_2=0,\\
			\frac{4L_2^2\beta-\alpha(e^{L_2}-1)^2r_o^2\mathcal{H}_o^2}{r_o^2L_2^2\left(1+\alpha\left(e^{2L_2}r_o^2+ \frac{(e^{L_2}-1)^2}{L_2^2}L_1^2\right)\right)}&L_2>0.
		\end{cases}
	\end{align*}
	
	\noindent If $\mathcal H_o$ and $P_o$ satisfy
	\[
		\mathcal{H}_o^2< \frac{4}{r_o^2}C_1,\text{ and }P_o^2< C_2x(r_o)^2,
	\] 
	then $(\Sigma,\gamma_o,H_o,P_o)$ admits a smooth collar extension $(M,g,K)$ with $M=[0,1]\times \mathbb{S}^2$,
	\begin{align*}
	g&=A^2\d s^2+u(Aks)^2r_o^{-2}\gamma(s),\\
	K&=D\left(A^2x'(u(Aks))\d s^2+x(u(Aks))u(Aks)r_o^{-2}\gamma(s)\right),
	\end{align*}
	that strictly satisfies the dominant energy condition.
\end{cor}
\begin{rem}\label{rem_collarconstruction_simplified}\,
 As $\overline{m}$ is chosen such that \eqref{eq_k^2eqD^2} holds, the Bartnik data directly affects the ODE for $u$. Moreover, $A$ is now fully determined by the Bartnik data and $L_1$, $L_2$.	
\end{rem}

\subsection{Gluing and bending lemma}\label{subsec_gluingbending}

We now state the main tools to attach the end of our collar to a suitable spherically symmetric graph in a Schwarzschild spacetime. Namely, we state a gluing lemma to attach two spherically symmetric initial data sets, and a bending lemma to push up the energy density in a small annular region. Both are straightforward modifications of the respective lemmas developed in \cite{MantoulidisSchoen2015-Apparent-Horizons} (see also \cite{CabreraPacheco2018-AH-Bartnik,CabreraPacheco2017-AF-CMC-Bartnik}) so we will state them now without proof. However, for the convenience of the reader, we include complete proofs in the appendix.

\begin{lem}[Gluing lemma]\label{lem-gluing}
	Let $f_i:[a_i,b_i]\to\mathbf \R^+$, where $i=1,2$, be smooth positive functions and let $\gamma_*$ be the standard round metric on $\mathbb{S}^{2}$. Consider metrics $g_i$ and symmetric $2$-tensors $K_i$ defined on $[a_i,b_i]\times \mathbb{S}^2$ by
	\begin{align*}
	g_i&=ds^2+f_i(s)^2\gamma_*\\
	K_i&=x'(f_i(s))ds^2+x(f_i(s))f_i(s)\gamma_*,
	\end{align*}
	for some $C^1$ function $x$. Assume that
	\begin{enumerate}[label=\roman*.]
		\item $\mu_i>0$ on $[a_i,b_i]$,
		\item $f_1(b_1)<f_2(a_2)$,
		\item $0<f_1'(b_1)<\sqrt{1+x(f(b_1))^2+2x(f(b_1))x'(f(b_1))f(b_1)}$,
		\item $f_2'(a_2)\le f_1'(b_1)$,
		\item $G_x(r)$ is monotone non-decreasing on $[f_1(b_1),f_2(a_2)]$.
	\end{enumerate}
	Then, after an appropriate translation of the intervals $[a_i,b_i]$, there exists a smooth positive function $f:[a_1,b_2]\to\mathbf \R^+$ such that
	\begin{enumerate}[label=\Roman*.]
		\item $f(s)=f_1(s)$ for $a_1\leq s\leq \frac{a_1+b_1}{2}$,
		\item $f(s)=f_2(s)$ for $\frac{a_2+b_2}{2}\leq s\leq b_2$,
		\item $\mu_f>0$ on $[a_1,b_2]$,
	\end{enumerate}  
	where $\mu_f$ denotes the energy density of the initial data set $(M,g,K)$, where $M=[a_1,b_2]\times \Sbb^2$,
	\begin{align}\begin{split}\label{eq-gluingform}
	g&=ds^2+f(s)^2\gamma_*\\
	K&=x'(f(s))ds^2+x(f(s))f(s)\gamma_*.
	\end{split}
	\end{align}
\end{lem}

\begin{rem}\label{rem_gluing}
	Note that the assumption that the second fundamental forms are given by
	\[
	K=x'(f_i(s))\d s^2+x(f_i(s))f_i(s)\gamma_*
	\]
	essentially reduces the problem to a metric gluing as in the time symmetric case (see \cite{CabreraPacheco2017-AF-CMC-Bartnik,CabreraPacheco2018-AH-Bartnik,MantoulidisSchoen2015-Apparent-Horizons}) as we take $x$ to be given and want to match two different functions $f_i$. It is natural to ask if one can achieve a similar gluing for different choices of $x$, and indeed this seems possible. However, for our purposes it is more straightforward to simply use the same choice of $x$ on both sides of the gluing construction. See Remark \ref{rem_schwarzschildgluing} (ii) for a brief comment.
\end{rem}

\begin{lem}[Bending lemma]\label{lem_bending}
	Let $(M,g,K)$ be an initial data set of the form $M=[a,\infty)\times \Sbb^2$, $a\ge 0$, 
	\begin{align*}
		g&=\d s^2+f(s)^2\gamma_*,\\
		K&=x'(f(s))\d s^2+x(f(s))f(s)\gamma_*,
	\end{align*}
	such that $\mu\ge \tau$ for some $\tau\in\R$, where $f$ is a smooth, positive function, $x$ is $C^1$, and $\gamma_*$ denotes the round metric on $\Sbb^2$.\newline
	For any $s_o>a$ satisfying $f'(s_o)>0$, there exists $\delta>0$ and a smooth positive function $\widetilde{f}$ on $[s_o-\delta,\infty)$ such that $\widetilde{f}(s)=f(s)$ for all $s\in[s_o,\infty)$ and the initial data set $(\widetilde{M},\widetilde{g},\widetilde{K})$, with $\widetilde{M}=[s_o-\delta,\infty)$,
	\begin{align*}
		\widetilde{g}&=\d s^2+\widetilde{f}^2(s)\gamma_*,\\
		\widetilde{K}&=x'(\widetilde{f}(s))\d s^2+x(\widetilde{f}(s))\widetilde{f}(s)\gamma_*,
	\end{align*}
	has $\widetilde{\mu}\ge \tau$ such that $\widetilde{\mu}>\tau$ on $[s_o-\delta,s_o)$. \newline
	In addition, if $f(s_o)>c>0$ and $f''(s_o)>0$, then $\delta$ can be chosen such that $\widetilde{f}(s_o-\delta)>c$, $\widetilde{f}'(s_o-\delta)<f'(s_o)$.
\end{lem}

\begin{rem}\label{rem_bending}
	We emphasise that the Bending Lemma requires no restrictions on $x$. In fact, the constant $\delta$ does not depend on the choice of $x$.
\end{rem}

\subsection{Gluing to a spherically symmetric graph in a Schwarzschild spacetime}\label{subsec_schwarzschildgluing}

Using the above tools, we are now ready construct the desired extension of the given Bartnik data. Before we prove the main theorem of this section, we first provide an intermediate result for gluing a spherically symmetric initial data set to a spherically symmetric graph in a suitable Schwarzschild spacetime. This will rephrase the criteria from the gluing lemma in terms of the Hawking mass at the end of the collar.

Recall that the Hawking mass for a surface $\Sigma$ is defined as
\[
	\m_{{H}}(\Sigma)=\sqrt{\frac{\btr{\Sigma}}{16\pi}}\left(1-\frac{1}{16\pi}\int_\Sigma\mathcal{H}^2_\Sigma\d\mu_\Sigma\right).
\]
In particular, if $\Sigma_s=\{s\}\times \Sbb^2$ is an initial data set $(M_f,g_f,K_f)$ of the form $M=[a,b]\times \Sbb^2$,
\begin{align*}
	g_f&=\d s^2+f(s)^2\gamma_*,\\
	K_f&=x'(f(s))\d s^2+x(f(s))f(s)\gamma_*,
\end{align*}
where $\gamma_*$ denotes the round metric, the Hawking mass $\m(s)=\m_{{H}}(\Sigma_s)$ of each leaf is given by
\begin{align}\label{eq_hawkingenergy}
	\m(s)=\frac{f(s)}{2}\left(1+x^2(f(s))-f'(s)^2\right).
\end{align}

We also briefly recall some basic properties of a spherically symmetric graph in a Schwarzschild spacetime. Recall from Section \ref{subsec_schwarzschild} that after a suitable change of coordinates, any spherically symmetric graph $(M_m,g_m,K_m)$ in the Schwarzschild spacetime (of mass $m>0$) can be written as $M_m=[0,\infty)\times \Sbb^2$, 
\begin{align*}
	g_m&=\d s^2+u_m^2(s)\gamma_*,\\
	K_m&=x'(u_m(s))\d s^2+x(u_m(s))u_m(s)\gamma_*,
\end{align*}
where $\gamma_*$ denotes the round metric on $\Sbb^2$, and $u_m\colon [0,\infty)\to [a,\infty)$ such that $u_m(0)=a$ and 
\[
	u_m'(s)=\sqrt{1-\frac{2m}{u_m(s)}+x^2(u_m(s))}
\]
with $1-\frac{2m}{a}+x^2(a)\ge 0$, and where $a$ is the radius of the inner boundary of the graph (in Schwarzschild coordinates). Note that independent of the choice of $x$, the largest radius $r_+\le 2m$ such that
\[
	1-\frac{2m}{r_+}+x^2(r_+)=0
\]
corresponds to an outermost minimal surface in $(M_m,g_m)$. However, as the outermost MOTS corresponds to the Schwarzschild radius $r=2m$, the respective initial data set $(M_m,g_m,K_m)$ will contain trapped surfaces with $\mathcal{H}^2<0$ if $r_+<2m$. Moreover, there is a plethora of $x$ such that
\[
	1-\frac{2m}{r}+x^2(r)>0
\]
on $(0,\infty)$. For these reasons, we will always choose $a=2m$ such that the inner boundary of $(M_m,g_m,K_m)$ will correspond to an outermost MOTS.

Finally, note that we have implicitly assumed that $u_m$ is well-defined on $[0,\infty)$, which is equivalent to 
\[
	\int\limits_{a}^r\frac{1}{\sqrt{1-\frac{2m}{\rho}+x^2(\rho)}}\d\rho <\infty
\]
for all $r\ge a$. While this is true for any $a$ as above, it is straightforward to check that it holds for $a=2m$ by comparison to the time symmetric case as $x^2\ge 0$.

It will be convenient to define the function
\begin{align}
	V_{x,m}(r)&=1-\frac{2m}{r}+x^2(r)\label{eq_G2}
\end{align}
for $m\in\mathbb R$. We are now ready to state our preliminary gluing result.

\begin{prop}[Gluing to Schwarzschild]\label{prop_gluingschwarzschild}\,\newline
	Let $f$ be a smooth positive function on $[a,b]$ with $f'(b)>0$. Let $(M_f,g_f,K_f)$ be an initial data set as above for a choice of $x$ such that 
	\begin{enumerate}
		\item[(i)] $\mu_f>0$,
		\item[(ii)] $f(b)>2\m(b)=2\m_{{H}}(\Sigma_b)$,
		\item[(iii)] $G_x$ is monotone non-decreasing on $[f(b),\infty)$,
		\item[(iv)] $V_{x,\m(b)}$ is strictly increasing on $[2\m(b),\infty)$.
	\end{enumerate}
	Then, for any $m\in (\m(b),\frac{1}{2}f(b))$ 
	there exists a spherically symmetric initial data set $(M,g,K)$ with $\mu\ge 0$ such that
	\begin{enumerate}
		\item[(I)] $(M,g,K)$ embeds isometrically into the Schwarzschild spacetime of mass $m$ outside of a neighbourhood of $\partial M$,
		\item[(II)] there exists a neighbourhood of $\partial M$ that is isometric to \newline$\left(\left[a,\frac{a+b}{2}\right]\times \Sbb^2,g_f,K_f\right)$.
	\end{enumerate}
\end{prop}

\begin{rem}\label{rem_schwarzschildgluing}\,
	\begin{enumerate}
		\item[(i)] From the definition of the Hawking mass, one quickly sees that $f(b)>2\m(b)$ is equivalent to $H^2-P^2>0$.
		\item[(ii)] Note that we can not guarantee that the resulting initial data set will have any physically meaningful asymptotic behavior for a general choice of function $x$. However, we may always achieve a particular asymptotic behavior by interpolation in an annular region in the Schwarzschild exterior, which will trivially preserve the DEC. One particular way to achieve this and preserve the spherical symmetry is by smoothly interpolating between two different choices of $x$, as $x$ fully determines the behavior of the graph in the Schwarzschild exterior (cf. \cite{Wolff2024-NEC-Umbilic}). For example, one may always achieve asymptotic flatness by simply multiplying $x$ with a suitable cut-off function to obtain an initial data set that is isometric to the time-symmetric Schwarzschild initial data outside a compact set. Asymptotically hyperboloidal extensions can be obtained similarly (cf. Remark \ref{rem_cmc}).

		\item[(iii)] We note that \emph{(iv)} above implies that $V_{x,m}$ is strictly increasing for all $m\ge \m(b)$ on $[2\m(b),\infty)$, that is, we have $u_m''>0$ in view of Lemma \ref{lem_bending}. Although a more careful argument shows that we do not require $u_m''>0$ and that condition \emph{(iv)} could be relaxed to only require that $V_{x,\m(b)}$ be increasing at $r=f(b)$, the presently chosen condition suffices for our purposes and is easier to work with.
		
		We also remark that the condition $V'_{x,\m(b)}(f(b))>0$ is equivalent to
		\begin{equation*}
			\m(b)>-f(b)^2x(f(b))x'(f(b)),
		\end{equation*}
		which agrees with the lower bound on the Hawking mass at the end of the collar assumed in both the asymptotically flat \cite{CabreraPacheco2017-AF-CMC-Bartnik} and asymptotically hyperbolic \cite{CabreraPacheco2018-AH-Bartnik} extensions, which correspond to $x(r)=0$ and $x(r)=r$ respectively.
	\end{enumerate}
\end{rem}

\begin{proof}
	We prove this by an application of the Gluing Lemma (Lemma \ref{lem-gluing}). To this end, for a given $m\in (\m(b),\frac{1}{2}f(b))$, we first aim to verify the following for a suitable $s_o\in(0,\infty)$:
	\begin{enumerate}
		\item[(a)] $f(b)<u_m(s_o)$,
		\item[(b)] $0<f'(b)<\sqrt{1+x^2(f(b))+2x(f(b))x'(f(b))f(b)}$,
		\item[(c)] $u_m'(s_o)\le f'(b)$,
		\item[(d)] $G_x$ is monotone non-decreasing on $[f(b),\infty)$.
	\end{enumerate}
	Note that (d) immediately follows from assumption (iii). It thus remains to verify (a)--(c). We first verify (b) by observing that rearranging Equation \eqref{eq_hawkingenergy} gives
	\begin{align}\label{eq_Hawking2}
		f'(b)^2=1-\frac{2\m(b)}{f(b)}+x^2(f(b)).
	\end{align}
	In particular, (b) is equivalent to
	\[
		0<\frac{2\m(b)}{f(b)^2}+2x(f(b))x'(f(b))=V'_{x,\m(b)}(f(b)),
	\]
	which is satisfied by assumptions \emph{(ii)} and \emph{(iv)}.\newline
	We now write
	\begin{equation*}
	u'_m(s)^2=1-\frac{2m}{u_m(s)}+x(u(s))^2
	\end{equation*}
	and imposing the condition $m>\m(b)$ gives
	\begin{equation}\label{eq-uprimebound}
	u'(s)^2<\left(1-\frac{f(b)}{u_m(s)}\right)+x(u(s))^2-\frac{f(b)}{u_m(s)}x(f(b))^2+\frac{f(b)}{u_m(s)}f'(b)^2.
	\end{equation}
	On the other hand, if $f(b)>u_m(0)=\max\left\{2m,0\right\}$ there exists $\widehat s\in(u_m(0),\infty)$ such that $u_m(\widehat s)=f(b)$. Note that \eqref{eq-uprimebound} then gives
	\begin{equation*}
	u_m'(\widehat s)^2<f'(b)^2
	\end{equation*}
	and since $u_m$ is increasing, choosing $s_o=\widehat s+\varepsilon$ for sufficiently small positive $\varepsilon$ ensures (a) and (c) hold.
	
	As $\mu_m=0$ on $(M_m,g_m,K_m)$ we can not directly apply Lemma \ref{lem-gluing}. Instead, we first apply Lemma \ref{lem_bending} to obtain an initial data set $(\widetilde{M},\widetilde{g},\widetilde{K})$ with $\widetilde{\mu}>0$, where $\widetilde{M}=[s_o-\delta,s_o)\times \Sbb^2$,
	\begin{align*}
		\widetilde{g}&=\d s^2+\widetilde{u}(s)^2\gamma_*,\\
		\widetilde{K}&=x'(\widetilde{u}(s))\d s^2+x(\widetilde{u}(s))\widetilde{u}(s)\gamma_*,
	\end{align*}
	and which extends (isometrically) to $([s_o,\infty)\times\Sbb^2,g_m,K_m)$. As $u_m(s_o)>f(b)>0$ and $u_m(s_o)''>0$ (by assumption \emph{(iv)}) we can choose $\delta>0$ such that (a)--(c) are satisfied with $\widetilde{u}$ at $s_o-\delta$. The claim then follows by applying Lemma \ref{lem-gluing} to $(M_f,g_f,K_f)$ and $(\widetilde{M},\widetilde{g},\widetilde{K})$.
\end{proof}

We now state our first main theorem.

\begin{thm}\label{thm_main1}\,\newline
	Let $(\Sigma,\gamma_o,H_o,P_o,0)$ be Bartnik data with $\gamma_o$ having strictly positive scalar curvature, where $H_o$, $P_o$ are constant with $H_o\ge \btr{P_o}>0$, and Hawking mass $\m_o\ge 0$.
	Let $x$ be a $C^1$ function on $[2\m_o,\infty)$ such that $x(r_o)\not=0$ and $\btr{x}(r)\le L_1+rL_2$ for some non-negative constants $L_1$, $L_2$, and that
	\begin{align*}
		G_x(r)=x^2+2x(r)x'(r)r,\,\text{ and }\;\; V_{x,\m_o}(r)=1-\frac{2\m_o}{r}+x^2(r)
	\end{align*}
	are monotone non-decreasing on $[2\m_o,\infty)$. Furthermore let
	\begin{align*}
		C_1&=
		\begin{cases}
			\frac{\beta}{\alpha}&L_2=0,\\
			\frac{\beta L_2^2}{\alpha(e^{L_2}-1)^2}&L_2>0,
		\end{cases}\\
		C_2&=
		\begin{cases}
			\frac{4\beta-\alpha r_o^2\mathcal{H}_o^2}{r_o^2\left(1+\alpha L_1^2\right)}&L_2=0,\\
			\frac{4L_2^2\beta-\alpha(e^{L_2}-1)^2r_o^2\mathcal{H}_o^2}{r_o^2L_2^2\left(1+\alpha\left(e^{2L_2}r_o^2+ \frac{(e^{L_2}-1)^2}{L_2^2}L_1^2\right)\right)}&L_2>0.
		\end{cases}
	\end{align*}
	If $\mathcal H_o$ and $P_o$ satisfy
	\begin{equation}\label{eq-HPcond}
		\mathcal{H}_o^2< \frac{4}{r_o^2}C_1,\text{ and } P_o^2< C_2x(r_o)^2,
	\end{equation}
	then there exists an initial data set $(M,g,K)$ that satisfies the DEC such that the Bartnik data $(\Sigma,\gamma_o, H_o,P_o,0)$ is realised as the inner boundary $\partial M$ and $(M,g,K)$ embeds isometrically into a Schwarzschild spacetime of mass $m>\m_o$ outside of a neighbourhood of $\partial M$.
\end{thm}

\begin{rem}\label{rem_main1}\,
	\begin{enumerate}
		\item[(i)] We note that $\m_o\geq0$ is equivalent to $\mathcal{H}^2_o\le 4r_o^{-2}$, which is the (square of) the spacetime mean curvature for round spheres in the Minkowski spacetime. In particular, the condition on $\mathcal{H}^2_o$ is trivially satisfied if $C_1>1$, for example, if $\gamma_o$ is sufficiently close to a round metric.
		\item[(ii)] Observe that under the assumptions of the theorem, $G_{Dx}$ is monotonically non-decreasing, and $V_{Dx,m}$ is strictly increasing for all $m\ge  \m_o$, $D\in[-1,1]$.
		\item[(iii)] Observe that $C_2$ itself depends on $P_o$ through $\mathcal{H}_o$. While \eqref{eq-HPcond} could be equivalently stated as a relation between $P_o$ and $H_o$, it seems natural to phrase it with respect to the spacetime mean curvature $\mathcal{H}_o$ in the non time-symmetric case. For example, we have $\mathcal{H}_o=0$ for MOTS data. Nevertheless, we note that \eqref{eq-HPcond} greatly simplifies for specific choices of $x$, in which case it is more convenient to state it in terms of $H_o$ and $P_o$. See, for example, Corollary \ref{kor_any} below.
	\end{enumerate}
\end{rem}

\begin{proof}
	By Corollary \ref{prop_collarextension_simplified} there exists a collar extension 
	with $\mu>0$ of the form
	$M=[0,1]\times \Sbb^2$,
	\begin{align*}
	g&=A^2\d t^2+u(Akt)^2r_o^{-2}\gamma(t),\\
	K&=D\left(A^2x'(u(Akt))\d t^2+x(u(Akt))u(Akt)r_o^{-2}\gamma(t)\right),
	\end{align*}
	where we note that $k^2=D^2\not=0$, as we have chosen the mass parameter in the collar such that \eqref{eq_k^2eqD^2} holds, and we take $D>0$ \footnote{That is, we replace $x$ by $-x$ if $P_ox(r_o)<0$, which does not affect the proof as all relevant computations only involve $x^2$ or its derivatives}. Moreover, it is easy to check that the assumptions directly imply that $D^2<C_2\le 1$ by Lemma \ref{lem_stmc}.
	By a change of coordinates $s=At$, we find $M=[0,A]\times \Sbb^2$,
	\begin{align*}
	g&=\d s^2+f(s)^2\widetilde{\gamma}(s),\\
	K&=\widetilde{x}'(f(s))\d s^2+\widetilde{x}(f(s))f(s)\widetilde{\gamma}(s),
	\end{align*}
	where $f(s)=u(ks)$, $\widetilde{x}(r)=Dx(r)$, and $\widetilde{\gamma}(s)=r_o^{-2}\gamma(\frac{s}{A})$. By construction, we have that $\widetilde{\gamma}(s)=\gamma_*$ (after possibly a change of coordinates via a fixed conformal diffeomorphism) on $((1-\widetilde{\varepsilon})A,A]$ for some fixed $\widetilde{\varepsilon}$. Hence, restricting our attention to this spherically symmetric part of the collar, the claim then follows once we verify the assumptions of Proposition \ref{prop_gluingschwarzschild} at the end of the collar.
	
	To this end, we first note that
	using \eqref{eq_hawkingenergy} and Lemma \ref{lem_stmc}, we find that 
	\begin{align}
	\begin{split}\label{eq_collarextension_hawkingmonotonicity}
	\frac{\d}{\d s}m_{\mathcal{H}}(\Sigma_s)&=Ak\frac{u'(Aks)}{2}\left(1-k^2-(k^2-D^2)G_x(u(ks))\right)\\
	&=Ak\frac{u'(Aks)}{2}(1-D^2)>0,
	\end{split}
	\end{align}
	as $k^2=D^2<1$ by construction, where $u'=\partial_t u$. As $f'>0$ by construction, we hence find that $f(A)>r_o\ge \m_o$ and $\m(A)>\m_o\geq0$. Thus it follows by the assumptions that $G_{Dx,\m(A)}$ is non-decreasing on $[f(A),\infty)$, and $V_{Dx,\m(A)}$ is strictly increasing on $[\m(A),\infty)$. Hence, to apply Proposition \ref{prop_gluingschwarzschild} it remains to show that $f(A)>2\m(A)$.
	
	However, this is always satisfied as \eqref{eq_collarextension_hawkingmonotonicity} directly implies that 
	\begin{align*}
		(f-2m)'(s)&=Au'(ADs)D^3>0,
	\end{align*}
	since $D=k>0$, and hence 
	\[
	f(A)-2\m(A)>f(0)-2\m_o=r_o-2\m_o\ge 0
	\]
	as $\mathcal{H}^2_o\ge 0$ by assumption. This concludes the proof, as the constructed extension has $\mu\ge 0$ everywhere which is equivalent to the DEC in this case by Corollary \ref{cor_DEC}.
\end{proof}

The next proposition further shows that the extension can always be chosen such that no trapped surfaces arise in its interior, which is a requirement for admissible Bartnik extensions.
\begin{prop}\label{prop_notrapped}
	Let $(\Sigma,\gamma_o,H_o,P_o,0)$  and $x$ be as in Theorem \ref{thm_main1}. Then the constructed initial data set $(M,g,K)$ can always be chosen such that it does not contain weakly trapped surfaces with $\mathcal{H}^2\le 0$. 
\end{prop}

\begin{proof}
	We first observe that it suffices to show that 
	\[
	\mathcal{H}^2_s>0
	\]
	for the leaves $\Sigma_s=\{s\}\times\Sigma$ of the canonical foliation for all $s> 0$. As the mean curvature $H_s$ of the leaves is positive by construction, we in particular then have $H_s>\btr{P_s}$ for all $s>0$. Using the strong maximum principle, this is sufficient to rule out the existence of weakly trapped surfaces by a standard argument.
	
	We now claim that 
	\[
	\mathcal{H}^2_s>0
	\]
	for all $s>0$. First, we note that
	\begin{align}\label{eq_prop_untrapped_eq1}
	\mathcal{H}^2_s=\frac{4}{f(s)^2}\left((f'(s))^2-D^2x(f(s))^2\right).
	\end{align}
	By continuity, it is straightforward to check that this strict inequality is preserved by the mollification in the Gluing Lemma \ref{lem-gluing} (if the mollification is chosen appropriately). Hence, it suffices to prove the claim individually in the collar region, the gluing region, and the exterior Schwarzschild region.
	
	As it is easy to check that
	\[
	f^3(s)\mathcal{H}^2_s=4(f(s)-2\m(s)),
	\]
	this implies the claim in the collar region, as $f(s)-2\m(s)$ is shown in the proof of Theorem \ref{thm_main1}.
	
	Moreover, as we are gluing onto an exterior Schwarzschild solution with mass $m<\tfrac12f(A)$ by construction, we have
	\[
		\mathcal{H}^2_s>0
	\]
	by the properties of the Schwarzschild spacetime as we are gluing away from the Schwarzschild horizon\footnote{This holds also for the slightly deformed Schwarzschild exterior that we actually glue to if we choose the deformation sufficiently small, which we may always do in Lemma \ref{lem_bending}}.
	
	It therefore remains to prove the claim in the gluing region prior to mollification. Recall from the proof of the Gluing Lemma (Lemma \ref{lem-gluing_appendix})\footnote{See Appendix \ref{App-A}.} that
	\[
		f(s)=\widehat{f}(s)=\int_{A}^s\zeta(\tau)\d\tau+f(A),
	\]
	in the gluing region $[A,s_o]$, where $f(A)=u(Ak)$ as in the proof of Theorem \ref{thm_main1}, and where $\zeta$ is a suitably chosen concave function such that $\zeta(A)=f'(A)$, $\zeta(s_o)=\widetilde{u}_m'(s_o)$ with $\widetilde{u}_{m}$, $s_o$ chosen as in Proposition \ref{prop_gluingschwarzschild}. More precisely, we have $s_o=\widetilde{s}+\delta$, where we may take any $\delta>0$ sufficiently small, with $\widetilde{s}$ such that
	\[
		\max\left(\widetilde{u}_{m}'(\widetilde{s}), \widetilde{u}_{m}'(s_o)\right)<f'(A)\text{ and }\widetilde{u}_m(\widetilde{s})=f(A)<\widetilde{u}_{m}(s_o).
	\]
	As $\zeta'\le 0$, in view of \eqref{eq_prop_untrapped_eq1} and the definition of $\widetilde{u}_m$ it thus suffices to show that
	\[
		\widetilde{u}'_m(s_o)^2=1-\frac{2m}{\widetilde{u}_m(s_o)}+D^2x(\widetilde{u}_m(s_o))^2>\sup\limits_{[\widetilde{u}_m(\widetilde{s}), \widetilde{u}_m(s_o)]} D^2x^2.
	\]
	Let $d$ be defined via
	\[
		3d=1-\frac{2m}{\widetilde{u}_m(\widetilde{s})}=1-\frac{2m}{f(A)}>0,
	\]
	which is positive as $f(A)>2m$ by construction, see Proposition \ref{prop_gluingschwarzschild}. By continuity, choose $\delta$ sufficiently small such that
	\begin{align*}
		\sup\limits_{[\widetilde{u}_m(\widetilde{s}), \widetilde{u}_m(s_o)]} D^2x^2&\le D^2x(\widetilde{u}_m(\widetilde{s}))^2+d\\
		1-\frac{2m}{\widetilde{u}_m(s_o)}+D^2x^2(\widetilde{u}_m(s_o))&\ge 1-\frac{2m}{\widetilde{u}_m(\widetilde{s})}+D^2x^2(\widetilde{u}_m(\widetilde{s}))-d=D^2x^2(\widetilde{u}_m(\widetilde{s}))+2d.
	\end{align*}
	Then, we find
	\[
	1-\frac{2m}{\widetilde{u}_m(s_o)}+D^2x^2(\widetilde{u}_m(s_o))-\sup\limits_{[\widetilde{u}_m(\widetilde{s}), \widetilde{u}_m(s_o)]} D^2x^2\ge d>0,
	\]
	which implies the claim and concludes the proof.
\end{proof}

We close this section by presenting two applications of Theorem \ref{thm_main1}. First, we identify criteria under which the data admits a non-time-symmetric constant mean curvature extension.

\begin{cor}\label{cor_cmc}
	Let $K_1$, $K_2$ be real constants not both equal to $0$. Then, any Bartnik data $(\Sigma,\gamma_o,H_o,P_o)$, with strictly positive scalar curvature, where $H_o$, $P_o$ are constant with $H_o\ge \btr{P_o}>0$, and Hawking mass $\m_o\ge 0$, satisfying
	\begin{enumerate}
		\item[(i)] $\mathcal{H}_o^2\le \frac{4K_2^2\beta}{r_o^2\alpha\left(e^{K_2}-1\right)^2}$,
		\item[(ii)] $P_o^2\le \frac{\left(4\beta K_2^2-\alpha{\left(e^{K_2}-1\right)^2}r_o^2\mathcal{H}_o^2\right)\left(K_2r_o-K_1r_o^{-2}\right)^2}{r_o^2 \left(K_2^2+\alpha\left(K_2^2e^{2K_2}r_o^2+\left(e^{K_2}-1\right)^2K_1^2r_o^{-4}\right)\right)}$,
		\item[(iii)] $K_1\not=K_2r_o^3$,
		\item[(iv)] $(K_1-2\m_o^3K_2)^2\leq4\m_o^4(1+9K_2^2\m_o^2)$,
	\end{enumerate}
	admits a smooth non-time-symmetric extension $(M,g,K)$ that satisfies the DEC, has constant mean curvature proportional to $K_2$, and contains no weakly trapped surfaces.
\end{cor}
\begin{rem}\label{rem_cmc}
	Note that for $K_2=0$ we obtain a maximal asymptotically flat initial data set and the conditions (i) and (ii) of Corollary \ref{cor_cmc} should be understood in the limiting sense. For $K_2>0$ we obtain asymptotically hyperboloidal initial data sets, that are umbilical and isometric to AdS-Schwarzschild outside of a compact set if $K_1=0$.
\end{rem}
\begin{proof}
	Consider $x(r)=K_2r-K_1r^{-2}$. Recall that this choice of $x$ corresponds to an initial data set with constant mean curvature $3K_2$ (see Remark \ref{rem_examples} (i)), where $x(r_o)=K_2r_o-K_1r_o^{-2}\not=0$ by assumption (iii). Observe that for this choice of $x$, $G_x$ is non-decreasing for any choice of $K_1$, $K_2$ and $V_{x,\m_o}$ is non-decreasing precisely when the polynomial
	\begin{equation*}
		f(r)=r^3\m_o+K_2^2r^6+K_1K_2r^3-2K_1^2
	\end{equation*}
	is non-negative on $[2\m_o,\infty)$. Since this is a quadratic polynomial in $r^3$ with one positive and one negative root, if we can impose conditions on $K_1$ and $K_2$ such that $f(2\m_o)\geq0$ then $f(r)\geq0$ for all $r\in[2\m_o,\infty)$. This amounts to ensuring
	\begin{equation*}
		-K_1^2+4K_2\m_o^3K_1+4\m_o^4+32+\m_o^6K_2^2\geq0,
	\end{equation*}
	which is equivalent to
	\begin{equation*}
		(K_1-2\m_o^3K_2)^2\leq4\m_o^4(1+9K_2^2\m_o^2).
	\end{equation*} 
	This is precisely assumption (iv).
	
	Taking $L_2=\btr{K_2}$, $L_1=r_o^{-2}\btr{K_1}$, we note that (i) and (ii) imply that
	\[
		\mathcal{H}_o^2\le \frac{4}{r_o^2}C_1,\text{ and } P_o^2\le \frac{4}{r_o^2}C_2x(r_o)^2.
	\]
	Hence, by applying Theorem \ref{thm_main1} we can construct a smooth extension $(M,g,K)$ with $\mu\ge 0$ of the form $M=[0,\infty)\times \Sbb^2$,
	\begin{align*}
		g&=\d s^2+f(s)^2 g(s),\\
		K&=Dx'(f(s))\d s^2+Dx(f(s))f(s)g(s),
	\end{align*}
	which can be chosen to contain no trapped surfaces by Proposition \ref{prop_notrapped}. As $D>0$ is uniquely determined by $P_o^2=4D^2x(r_o)^2r_o^{-2}$, we observe that $(M,g,K)$ has constant mean curvature $\frac{3r_o^3\btr{P_o}K_2}{2\btr{K_2r_o^3-K_1}}$.
\end{proof}

Note that we have phrased the statement of Corollary \ref{cor_cmc} such that we first make a choice of $x$ which then restricts the Bartnik data that admit such an extension purely for convenience sake, as one in general wants to choose $x$, in this case the parameters $K_1$, $K_2$, with respect to given Bartnik data. 

We further give criteria when the Bartnik data admits any extension, which seem to be the mildest restriction on the data that allow our construction to work.

\begin{cor}\label{kor_any}
		Let $(\Sigma,\gamma_o,H_o,P_o,0)$ be Bartnik data with $\gamma_o$ having strictly positive scalar curvature, where $H_o$, $P_o$ are constant with $H_o\ge \btr{P_o}>0$, and Hawking mass $\m_o\ge 0$. If 
		\begin{align}\label{eq-corcond}
			\mathcal{H}_o^2&<\frac{4}{r_o^2}\frac{\beta}{\alpha},\text{ and }
			P_o^2<\frac{4\beta-\alpha r_o^2\mathcal{H}_o^2}{\alpha r_o^2},
		\end{align}
		then $(\Sigma,\gamma_o,H_o,P_o,0)$ admits an extension $(M,g,K)$ that satisfies the DEC and contains no weakly trapped surfaces.
\end{cor}
\begin{rem}\label{rem_any}\,
	\begin{enumerate}
		\item[(i)] We state condition \eqref{eq-corcond} as is, to better compare to Theorem \ref{thm_main1}. However, it is in fact equivalent to \eqref{eq-introAcond} appearing in Theorem A of the introduction.
		\item[(ii)] Our choice of $x$ in the collar extension seems to be optimal in terms of the constants $C_1$, $C_2$ in Remark \ref{rem_collarconstruction_simplified}. Recall from Remark \ref{rem_schwarzschildgluing} that $x$ may be varied in any way in the exterior Schwarzschild region, so the asymptotics of the constructed extension can be chosen as desired.
	\end{enumerate}
\end{rem}
\begin{proof}
	Choose $x$ constant, $x\equiv L_1$, $L_2=0$. Observe that
	\[
		\mathcal{H}_o^2\le \frac{4}{r_o^2}C_1,
	\]
	and $G_x$, $V_{x,\m_o}$ are non-decreasing on $[2\m_o,\infty)$ for any $L_1\not=0$. Finally, choose $\btr{L_1}$ sufficiently large such that
	\[
			P_o^2<\frac{4\beta-\alpha r_o^2\mathcal{H}_o^2}{(1+\alpha L_1^2) r_o^2}L_1^2=C_2x(r_o)^2.
	\]
	The claim then follows from Theorem \ref{thm_main1} and Proposition \ref{prop_notrapped}.
\end{proof}

\section{Estimating the Bartnik mass from extensions}\label{sec-estimate}

For any choice of the function $x$ in the construction of extensions in the preceding section, we find that the mass of the extension is equal to the Hawking mass at the end of the collar. In particular, we obtain different estimates for each choice of $x$. However, just as in the time-symmetric case, we do not expect that there is any optimal choice for this. For this reason, we elect to only present one example of this that we believe demonstrates nice qualitative properties. However, we emphasise that one can obtain many different estimates using this construction.

Taking the initial data of the form \eqref{eq_ansatz_collar}, we set $x(r)=\frac{B}{\sqrt{r}}$ for some constant $B\neq0$, as this satisfies $G_x\equiv0$ and therefore simplifies \eqref{eq_dec_collar} considerably. In particular, the last term in \eqref{eq_dec_collar} vanishes for any choice of parameter $\overline{m}$, so we can obtain a collar without fixing the value of $\overline{m}$ as was done in the previous section. We simply require
\begin{equation}\label{eq-DEC-simple}
	2\beta-2k^2-\frac{u^2\alpha}{A^2}\geq0,
\end{equation}
and therefore the collar used for this estimate is slightly different (and simpler) than the general one constructed in the preceding section. Note that since we fix a choice of $x$, it is easier to absorb the constant $D$ from our ansatz into the constant $B$. That is, take $D=1$ and therefore $B=\tfrac12 P_o r_o^{3/2}$.

For this choice of $x$, $u$ satisfies
\begin{equation*}
	u'=\sqrt{1+\frac{B^2-2\overline{m}}{u}},
\end{equation*}
where $\overline{m}$ remains a free parameter. A clean choice we therefore now make is to set $\overline{m}=\tfrac12 B^2$, so $u(Aks)= r_o+Aks$ on $[0,1]$. In this case, to satisfy \eqref{eq-DEC-simple} we may simply ask for
\begin{equation}
	\beta-k^2-\frac{r_o^2+A^2k^2}{A^2}\alpha >0.
\end{equation}
This is easily achieved by choosing
\begin{align}\begin{split}\label{eq-Achoice}
	A&=r_o\sqrt{\alpha}\left( \beta-(\alpha+1)k^2 \right)^{-1/2}\\
	&=r_o\sqrt{\alpha}\left( \beta-(\alpha+1)\left( \frac{H_o^2r_o^2}{4} \right) \right)^{-1/2}\\
	&=r_o\,\sqrt{ \frac{\alpha}{\beta-(\alpha+1)\frac{H_o^2r_o^2}{4}}}\end{split}
\end{align}
provided 
\begin{equation}\label{eq-simple-hold-requirement}
\beta>(\alpha+1)\frac{H_o^2r_o^2}{4}.
\end{equation}
A quick calculation  shows that
\begin{align*}
	\mathcal{H}^2(s)&=H(s)^2-P(s)^2=\frac{H_o^2r_o^2}{u^2}-\frac{r_o^3}{u^3}P_o^2\\
	&> \frac{r_o^2}{u^2}\left( H_o^2-P_o^2\right)\ge0,
\end{align*}
for $s>0$, where we write $u=u(Aks)$ above, which excludes MOTS in the collar by a standard argument as before.

We then find that the Hawking mass at the end of the collar is given by
\begin{align*}
	\m_H(1)&=\frac{u(Ak)}{2}\left( 1- \frac{u(Ak)^2}{4}\mathcal H(1)^2 \right)\\
	&=\frac{u(Ak)}{2}\left( 1- \frac{u(Ak)^2}{4}\left(\frac{H_o^2r_o^2}{u(Ak)^2}-\frac{r_o^3}{u(Ak)^3}P_o^2 \right) \right)\\
	&=\frac{u(Ak)}{2}\left(   1-\left( \frac{H_o^2r_o^2}{4}-\frac{P_o^2r_o^2}{4}\frac{r_o}{u(Ak)}  \right)   \right)\\
	&\leq \m_H(0)+ \frac{Ak}{2}\left(1-\frac{H_o^2r_o^2}{4}\right),
\end{align*}
where $A$ is given by \eqref{eq-Achoice}. Note that as either $\alpha$ or $H_o$ tend to zero, so does $A$. Substituting the values for $A$ and $k$ gives
\begin{equation}
	\m_H(1)\leq \m_H(0)+\frac{H_or_o^2\sqrt{\alpha}(4-H_o^2r_o^2)}{8\sqrt{4\beta-(\alpha+1)H_o^2r_o^2}}.
\end{equation}

In order to obtain a Bartnik mass estimate from this simplified collar, we must also ensure that condition \emph{(iv)} of Proposition \ref{prop_gluingschwarzschild} holds, which is equivalent to $\m_H(1)> \tfrac12 B^2=\tfrac{P_o^2r_o^3}{8}$. One can quickly check from the above calculation for the Hawking mass that this is equivalent to asking
\begin{equation}\label{eq-mneginftything}
	\frac{H_o^2r_o^2}{4}<1,
\end{equation}
which follows directly from the assumption \eqref{eq-simple-hold-requirement}.
\begin{rem}
	The choice $x(r)=\frac{B}{\sqrt r}$ does not correspond to an asymptotically flat initial data set in the Schwarzschild spacetime, so after this gluing we obtain initial data that is not asymptotically flat. However, as mentioned in Remark \ref{rem_schwarzschildgluing} (ii), it is clear that in the exterior region we can interpolate between the constructed slice of Schwarzschild and any asymptotically flat slice in the same Schwarzschild spacetime.
\end{rem}
We have now established the following estimate for the Bartnik mass.
\begin{thm}\label{thm-mb-bound}
	Let $(\Sigma, \gamma,H_o,P_o,0)$ be Bartnik data with $\gamma$ having strictly positive scalar curvature, where $H_o$, $P_o$ are constant with $H_o\ge \btr{P_o}>0$, satisfying
		\begin{equation*}
			\frac{H_o^2r_o^2}{4}<\frac{\beta}{\alpha+1}.
		\end{equation*}
	Then the Bartnik mass of $(\Sigma, \gamma,H_o,P_o,0)$ satisfies
	\begin{equation*}
		\m_B(\Sigma, \gamma,H_o,P_o,0)\leq \m_H(0)+\frac{H_or_o^2\sqrt{\alpha}(4-H_o^2r_o^2)}{8\sqrt{4\beta-(\alpha+1)H_o^2r_o^2}}.
	\end{equation*}
\end{thm}

\begin{rem}
	Since the choice $x(r)=\frac{B}{\sqrt r}$ gives $u'(s)=\sqrt{1-\frac{2\overline{m}-B^2}{u(s)}}$, the construction above in this section is essentially identical to the time-symmetric case with mass parameter equal to $(\overline{m}-\tfrac{B^2}{2})$. It is therefore possible to use this choice to obtain a variety of very similar estimates following the different estimates in the time-symmetric case (for example, \cite{CabreraPacheco2017-AF-CMC-Bartnik,MiaoPiubello2024,MiaoWangXie2020}).
\end{rem}

\section{Reduction to time-symmetric estimates}\label{sec-reduct}
In this section we give an alternate method for constructing estimates for the Bartnik mass outside of symmetry, by connecting non-time-symmetric initial data to time-symmetric initial data. Specifically, we prove Theorem \ref{thm-mass-est-2} from the Introduction.

This section closely follows the work of Lin \cite{Lin2024-Spin-Creased} in his PhD thesis. We begin with a general method for constructing initial data on $\Sbb^2\times[0,1]$ with prescribed Bartnik boundary data $(\gamma,H_o,P_o,\omega^\perp=0)$. We consider the ansatz
\begin{align}\label{eq-lin_ansatz}
	g&=\varepsilon^2\d s^2+e^{2\varepsilon f}\gamma,\\
	K&=\varepsilon^{2}b \d s^2+ae^{2\varepsilon f}\gamma,
\end{align}
where $\varepsilon$ is a small positive constant, and $f,a,b$ are smooth functions on $[0,1]$.

Along the leaves $\Sigma_s$ of constant $s$, we have $H=2f'$ and $P=\trace_\Sigma(K)=2a$. The Bartnik data then imposes
\begin{align*}
	f(0)=0,\qquad
	f'(0)=\frac12 H_o,\qquad
	a(0)=\frac12 P_o.
\end{align*}
Note that $b$ remains freely specifiable.
We next record some quantities that will be useful.
\begin{align*}
	A_{AB}=f'e^{2\varepsilon f}g_{AB}=\frac12 He^{2\varepsilon f}g_{AB},\qquad\\
	\trace_g(K)=2a+b,\qquad\qquad
	|K|_g^2=2a^2+b^2,
\end{align*}
where $A$ is the second fundamental form of $\Sigma_s$. Using the above and the evolution formula for $H$,
\begin{equation*}
	H'=-\varepsilon(|A|_\Sigma^2+\ric_{\nu\nu}),
\end{equation*}
the scalar curvature of $g$ can be computed as
\begin{equation}
	R_g=e^{-2\varepsilon f}R_\gamma-\frac32 H^2-2\varepsilon^{-1}H'.
\end{equation}
Following \cite{Lin2024-Spin-Creased}, the energy and momentum densities are calculated directly as
\begin{align*}
	\mu&=\frac12e^{-2\varepsilon f}R_\gamma-\frac34 H^2-\varepsilon^{-1}H'+a^2+2ab,\\
	J_\nu&=2f'(b-a)-2\varepsilon^{-1}a',\\
	J_A&=0.
\end{align*}
In particular, we can force $J_\nu$ to vanish by choosing 
\begin{equation}
	b=a+\frac{a'}{\varepsilon f'}.
\end{equation}
The DEC now reduces to the non-negativity of $\mu$, which is now given as
\begin{equation}
	\mu=\frac12e^{-2\varepsilon f}R_\gamma-\frac34 H^2-\varepsilon^{-1}H'+3a^2+\frac{2aa'}{\varepsilon f'}.
\end{equation}
Although we could write $H$ and $P$ in terms of $f$ and $a$, it is convenient to keep $H$ and $P$ terms. We then have
\begin{equation}\label{eq-muHPconst}
	\mu=\frac12e^{-2\varepsilon f}R_\gamma-\frac34 (H^2-P^2)+\frac{2aa'-H'f'}{\varepsilon f'}.
\end{equation}
Note that the Hawking mass \eqref{eq-Hawkingmassdefn} along this collar will remain close to constant provided $\mathcal H^2 =H^2-P^2$ remains constant, since $\varepsilon$ controls the area growth and this will be chosen small. To this end, we impose $\mathcal H^2=H^2-P^2=C^2$ for some constant $C$, and differentiate this to obtain a relation between $a$ and $f$,
\begin{equation}
	2HH'=2PP'=8aa'=4f'H'=8f'f''.
\end{equation}
If $a$ and $f$ can be chosen to satisfy $aa'=f'f''$ then \eqref{eq-muHPconst} reduces to
\begin{equation}
	\mu=\frac12e^{-2\varepsilon f}R_\gamma-\frac34 C^2.
\end{equation}
This is simply achieved by first fixing $a(t)$ then defining
\begin{equation*}
	f(t)=\int_0^t \sqrt{a(\tau)^2+\frac{C^2}{4}}\,d\tau.
\end{equation*}
In particular, we can choose $a$ (and therefore also $f''$) to be identically zero on the interval $[\tfrac12,1]$ so that $P(s)$ vanishes (and $H(s)$ is constant) on $[\tfrac12,1]$.

Then we conclude that for $\varepsilon$ taken sufficiently small we need only ask that the Bartnik data satisfies
\begin{equation*}
	R_\gamma>\frac32 C^2=\frac32 \mathcal H^2
\end{equation*}
to construct a metric of the form \eqref{eq-lin_ansatz} on $[0,1]\times S^2$ such that the Bartnik data for $\Sigma_0$ is the given data, while the Bartnik data for $\Sigma_1$ is the Bartnik data $(\Sigma,(1+\delta)g,\mathcal H_o,0,0)$ for any given $\delta>0$. That is, we obtain time-symmetric Bartnik data with mean curvature equal to $\mathcal H_o$ and Hawking mass arbitrarily close to the Hawking mass of the initial Bartnik data. Furthermore, since $\mathcal H_o$ is constant along the leaves, trapped surfaces are excluded in the collar as before by the maximum principle.

More concretely, we have established the following.
\begin{thm}\label{thm_main2}
	Let $(\S,\gamma,H_o,P_o,0)$ be Bartnik data with $\gamma$ having strictly positive scalar curvature, where $H_o$, $P_o$ are constant with $H_o\ge \btr{P_o}>0$, satisfying
	\begin{equation}\label{eq-reduction-condition}
		R_\gamma>\frac32\mathcal H_o^2.
	\end{equation}
	Then there exists initial data $(\gamma,K)$ on $\S\times[0,1]$ satisfying the DEC and containing no MOTS, such that $(\S,\gamma,H_o,P_o,0)$ is the induced Bartnik data at $t=0$ and $(\S,(1+\epsilon)^2\gamma,\mathcal H_o,0,0)$ is the induced Bartnik data at $t=1$, for any $\epsilon>0$.
\end{thm}

One may expect that then any admissible extension of the time-symmetric data $(\S,(1+\epsilon)^2\gamma,\sqrt{H_o^2-P_o^2},0,0)$ yields an admissible extension to the Bartnik data $(\S,\gamma,H_o,P_o,0)$, provided \eqref{eq-reduction-condition} holds. In which case, one should then be able to conclude
\begin{equation*}
	\m_B(\S,\gamma,H_o,P_o,0)\leq \m_B(\S,\gamma,\mathcal H_o,0,0).
\end{equation*}
However, there is a subtlety here that prevents us from drawing that conclusion. Namely that the standard corner-smoothing constructions affect the entire manifold in such a way that cannot ensure that MOTS are not created in the process, which illustrates also the need for a smooth gluing lemma such as Lemma \ref{lem-gluing_appendix}. A detailed discussion of this subtlety is given by Jauregui \cite{jauregui2019smoothing} for example. One could of course mildly alter the definition of Bartnik mass to use a stronger non-degeneracy condition (as in \cite{jauregui2019smoothing}) or to allow extensions of low regularity rather than smooth ones, and then draw this conclusion still. Alternatively, we could use known extensions to time-symmetric initial data that allow for smooth gluing while preserving the MOTS condition. For example, taking $\alpha$ and $\beta$ as defined above, we have the following Corollary.
\begin{cor}\label{cor-massbound2}
	Let $(\S,\gamma,H_o,P_o,0)$ be Bartnik data with $\gamma$ having strictly positive scalar curvature, where $H_o$, $P_o$ are constant with $H_o\ge \btr{P_o}>0$, satisfying
	\begin{equation}
		R_\gamma>\frac32\mathcal H_o^2.
	\end{equation}
	Then if 
		\begin{equation*}
			\frac{\mathcal H_o^2r_o^2}{4}<\frac{\beta}{1+\alpha},
		\end{equation*}
		then
		\begin{equation*}
			\m_B(\S,\gamma,H_o,P_o,0)\leq\left( 1+ \left(\frac{\alpha r_o^2\mathcal H_o^2}{4\beta-(1+\alpha)r_o^2\mathcal H_o^2}\right)^{1/2} \right)\m_H(\S,\gamma,H_o,P_o,0).
		\end{equation*}
\end{cor}
\begin{proof}
	Let 
	\[
	m_*=\left( 1+ \left(\frac{\alpha r_o^2\mathcal H_o^2}{4\beta-(1+\alpha)r_o^2\mathcal H_o^2}\right)^{1/2} \right)\m_H(\S,\gamma,H_o,P_o,0),
	\]
	which is the upper bound for the Bartnik mass of the time-symmetric Bartnik data $(\S,\gamma,\mathcal H_o,0,0)$ given by Theorem 1.1 of \cite{CabreraPacheco2017-AF-CMC-Bartnik}. This upper bound is obtained by constructing admissible extensions with mass arbitrarily close to $\m_*$, so we let $(M,g_\epsilon)$ be such an extension of the Bartnik data $(\S,\gamma,\mathcal H_o,0,0)$ constructed in the proof, with ADM mass $\m=(1+\epsilon)m_*$. Note that the rescaled metric $(1+\epsilon)^2g_\epsilon$ has ADM mass $(1+\epsilon)^2m_*$ and boundary Bartnik data $(\S,(1+\epsilon)^2\gamma,(1+\epsilon)^{-1}\mathcal H_o,0,0)$. We may assume without loss of generality that the path of metrics $\gamma(s)$ defining $\alpha$ and $\beta$ is constant on some interval containing $s=0$ since any path that does not satisfy this may be approximated by a path that does with the corresponding $\alpha$ and $\beta$ arbitrarily close to the original path. This is then precisely the situation for the smooth gluing to apply (Lemma \ref{lem-gluing} with $x\equiv0$ or Lemma 2.1 of \cite{CabreraPacheco2017-AF-CMC-Bartnik}, for example). That is we can smoothly glue the collar constructed in Theorem \ref{thm_main2} to $(M,g_\epsilon)$ while preserving the DEC and without introducing any MOTS, providing an admissible extension of $(\S,\gamma,H_o,P_o,0)$ with ADM mass $\m=(1+\epsilon)^2m_*$. Taking $\epsilon\to0$ completes the proof.
\end{proof}
\begin{rem}
	Corollary \ref{cor-massbound2} is obtained using the time-symmetric estimate from \cite{CabreraPacheco2017-AF-CMC-Bartnik} but any time-symmetric estimate using a similar collar construction could be used instead. The only requirement is that the extensions can be constructed in such a way that in a neighbourhood of the boundary the metric can be expressed as
	\begin{equation*}
		ds^2+f(s)^2\gamma_o
	\end{equation*}
	for a fixed metric $\gamma_o$ on $\S$ with positive Gauss curvature.
\end{rem}

\begin{rem}\label{rem-nonzerow}
	Throughout this work we restricted our attention to the case $\omega^\perp\equiv0$, so we conclude by briefly remarking on the case of non-zero $\omega^\perp$. One can easily alter the ansatz for $K$ to include a $K_{sJ}dsdx^J$ cross term that prescribes $\omega^\perp$, and indeed carry through similar computations to those presented here. However, the expressions obtained become cumbersome quickly if one attempts to maintain the same generality of the kinds of extensions constructed. For this reason, we reserve estimates of the Bartnik mass with non-zero $\omega^\perp$ for future work.
\end{rem}

\appendix

\section{Gluing and bending lemma}\label{App-A}

	\begin{lem}[Gluing lemma]\label{lem-gluing_appendix}
		Let $f_i:[a_i,b_i]\to\mathbb R^+$, where $i=1,2$, be smooth positive functions and let $\gamma_*$ be the standard round metric on $\mathbb S^{2}$. Consider metrics $g_i$ and symmetric $2$-tensors $K_i$ defined on $[a_i,b_i]\times \mathbb S^2$ by
		\begin{align*}
		g_i&=ds^2+f_i(s)^2\gamma_*\\
		K_i&=x'(f_i(s))ds^2+x(f_i(s))f_i(s)\gamma_*,
		\end{align*}
		for some $C^1$ function $x$. Assume that
		\begin{enumerate}[label=\roman*.]
			\item $\mu_i>0$ on $[a_i,b_i]$,
			\item $f_1(b_1)<f_2(a_2)$,
			\item $0<f_1'(b_1)<\sqrt{1+x(f_1(b_1))^2+2x(f_1(b_1))x'(f_1(b_1))f_1(b_1)}$,
			\item $f_2'(a_2)\le f_1'(b_1)$,
			\item $G_x(r)$ is monotone non-decreasing on $[f_1(b_1),f_2(a_2)]$.
		\end{enumerate}
		Then, after an appropriate translation of the intervals $[a_i,b_i]$, there exists a smooth positive function $f:[a_1,b_2]\to\mathbb R^+$ such that
		\begin{enumerate}[label=\Roman*.]
			\item $f(s)=f_1(s)$ for $a_1\leq s\leq \frac{a_1+b_1}{2}$,
			\item $f(s)=f_2(s)$ for $\frac{a_2+b_2}{2}\leq s\leq b_2$,
			\item $\mu_f>0$ on $[a_1,b_2]$,
		\end{enumerate}  
		where $\mu_f$ denotes the energy density of the initial data set $(M,g,K)$, where $M=[a_1,b_2]\times \Sbb^2$,
		\begin{align}\begin{split}\label{eq-gluingform_appendix}
		g&=ds^2+f(s)^2\gamma_*\\
		K&=x'(f(s))ds^2+x(f(s))f(s)\gamma_*.
		\end{split}
		\end{align}
	\end{lem}

	\begin{proof}
		This proof follows the same ideas as the gluing lemmas in \cite{MantoulidisSchoen2015-Apparent-Horizons,CabreraPacheco2017-AF-CMC-Bartnik,CabreraPacheco2018-AH-Bartnik}. 
		
		First note that the assumptions $(ii)$ and $(iii)$ ensure that the intervals can always be translated so that
		\begin{align}\begin{split}\label{eq-translatedint}
		(a_2-b_1)f_1'(b_1)=f_2(a_2)-f_1(b_1),&\qquad \text{ if }f_1'(b_1)=f_2'(a_2),\\
		(a_2-b_1)f_1'(b_1)>f_2(a_2)-f_1(b_1)>(a_2-b_1)f_2'(b_2),&\qquad \text{ if }f_1'(b_1)>f_2'(a_2)
		\end{split}.
		\end{align}
		This ensures that we can interpolate between $f_1'(b_1)$ and $f_2'(a_2)$ with a smooth function $\zeta$ on $[b_1,a_2]$ satisfying $\zeta(b_1)=f_1'(b_1)$, $\zeta(a_2)=f'_2(a_2)$, $\zeta'\leq0$ and $\int_{b_1}^{a_2}\zeta(s)\,ds=f_2(a_2)-f_1(b_1)$. We then define $\widehat f$ on $[a_1,b_2]$ by
		\[
		\widehat f(s) = 
		\begin{cases}
		f_1(s) & \text{ if  }s\in[a_1,b_1], \\
		f_1(b_1)+\int_{b_1}^s\zeta(t)\,dt. & \text{ if  }s\in[b_1,a_2],\\
		f_2(s) & \text{ if  }s\in[a_2,b_2].
		\end{cases}
		\]
		By construction $\widehat f$ is $C^{1,1}$ on $[a_1,b_2]$ and $C^2$ everywhere except at $b_1$ and $a_2$. Note that  a metric $\gamma$ and tensor $K$ of the form \eqref{eq-gluingform} satisfy $\mu_f>0$ if and only if
		\begin{equation}\label{eq-fbis}
		f''(s)<\frac{1}{2f(s)}\left( 1-f'(s)^2+x(f(s))^2+2x(f(s))x'(f(s))f(s) \right).
		\end{equation}
		We define $\Omega[f]$ from the right-hand side of \eqref{eq-fbis} as
		\begin{equation}
		\Omega[f](s)=\frac{1}{2f(s)}\left( 1-f'(s)^2+x(f(s))^2+2x(f(s))x'(f(s))f(s) \right).
		\end{equation}
		By assumption $\Omega[f_i]-f''_i>0$ on each $[a_i,b_i]$. We now show that $\Omega[\widehat f](s)-\widehat f''(s)>0$ on $(b_1,a_2)$.

		Note that $\widehat f''\leq0$ on $(b_1,a_2)$, and
		\begin{align*}
		\widehat f'(s)\leq f'(b_1)&<\sqrt{1+x(f(b_1))^2+2x(f(b_1))x'(f(b_1))f(b_1)}\\
		&\leq\sqrt{1+x(f(s))^2+2x(f(s))x'(f(s))f(s)},
		\end{align*}
		which follows from the fact that $G$ is non-decreasing. That is, $\Omega[\widehat f]$ is positive on $[a_1,b_2]$ while $\widehat f''\leq0$ on $[a_1,b_2]\setminus\{b_1,a_2\}$. 
		
		We now smooth out $\widehat{f}$ using a standard argument, exactly as in \cite{CabreraPacheco2017-AF-CMC-Bartnik,CabreraPacheco2018-AH-Bartnik}, repeating the argument verbatim only for the sake of exposition.
		
		Let $\delta >0$ satisfy
		\begin{align*}
		\delta < b_1 - \frac{a_1+b_1}{2} \, \text{and} \,  \delta < \frac{a_2+b_2}{2}-a_2 
		\end{align*}
		and let $\eta_\delta$ be a smooth cut-off function equal to $1$ on $[b_1-\delta,a_2+\delta]$, vanishing on the set $\left[a_1,\frac{a_1+b_1}{2} \right]\cup \left[\frac{a_2+b_2}{2},b_2\right]$, and satisfying $ 0 <  \eta_{\delta}   < 1 $ elsewhere. Let $\phi:\mathbb{R}\to [0,\infty)$ be a standard smooth mollifier with compact support in $[-1,1]$ and $\int_{-\infty}^\infty\phi(t)\,dt=1$.
		
		For any $\varepsilon\in(0,\frac{\delta}{4})  $, define $f_{\varepsilon}$ by 
		\begin{align}
		f_{\varepsilon}(s)= \int_{-\infty}^\infty \widehat{f}(s-\varepsilon \eta_{\delta}(s)t)\phi(t)\, dt \quad \text{ for }\quad  s \in [a_1, b_2].  
		\end{align}
		Note that $f_{\varepsilon}$ is smooth on $[a_1,b_2]$, $f_\varepsilon\equiv \widehat{f}$ on $\left[a_1,\frac{a_1+b_1}{2} \right]\cup \left[\frac{a_2+b_2}{2},b_2\right]$ and 
		\begin{align}
		f'_{\varepsilon}(s)=\int_{-\infty}^\infty \widehat{f}' (s-\varepsilon \eta_{\delta}(s)t)(1 - \varepsilon \eta_\delta'(s) t  ) \phi(t)  \,  dt  \quad \forall\, s \in [a_1, b_2] .
		\end{align}
		Since $ \widehat{f}'$ is $C^0$ everywhere and $C^1$ on $[a_1,b_2]\setminus\{b_1,a_2\}$, standard mollification arguments give for $s \in (b_{1}-\delta,a_{2}+\delta)$
		\begin{align}\label{eq-mollfbis1}
		f''_{\varepsilon}(s) = & \  \frac{d}{ds} \left(  \int_{-\infty}^\infty \widehat{f}' (s-\varepsilon t) \phi(t)  \,  dt \right) 
		=  \int_{-\infty}^\infty \widehat{f}''  (t) \phi_\varepsilon ( s -t)   \,  dt, 
		\end{align}
		using $\eta_{\delta}(s)=1$, where $ \phi_\varepsilon(s) = \frac{1}{\varepsilon} \phi ( \frac{ s}{\varepsilon} )$. Moreover, for $s \in [a_{1},b_{2}]\setminus[b_1 -  \frac{1}{4} {\delta}, a_2 + \frac{1}{4}  \delta] $ where $\widehat{f}$ is smooth, we have
		\begin{align}\label{eq-mollfbis2} 
		\begin{split}
		f''_{\varepsilon}(s) = & \ \int_{-\infty}^\infty \widehat{f}'' (s-\varepsilon \eta_{\delta}(s)t)\,(1 - \varepsilon \eta_\delta'(s) t  )^2 \phi(t)  \,  dt  \\ 
		& \ - \varepsilon \int_{-\infty}^\infty \widehat{f}' (s-\varepsilon \eta_{\delta}(s)t)\,  \eta_\delta'' (s) t   \phi(t)  \,  dt.
		\end{split} 
		\end{align}
		We define $d>0$ by
		\[
		3d=\inf_{s\in [a_1,b_2]\setminus\{b_1,a_2\}}\left( \Omega[\widehat f](s)-\widehat f''(s) \right). 
		\]
		Since $\widehat f\in C^2\left([a_1,b_2]\setminus\{b_1,a_2\}\right)$ we have
		\[
		\widehat f''(s)\leq \Omega[\widehat f](s)-3d
		\]
		on that same set (everywhere $\widehat f''$ is defined).
		
		Now for sufficiently small $\varepsilon$, 
		\begin{equation}
		\sup_{s\in[a_1,b_2]}|\Omega [\widehat f](s)-\Omega[f_\varepsilon](s)|<d,
		\end{equation}
		since $f_\varepsilon\to \widehat f$ in $C^1$, and from \eqref{eq-mollfbis1} and \eqref{eq-mollfbis2} we have
		\begin{equation}
		f''_\varepsilon<\sup_{|t-s|}\widehat f''(t)+d.
		\end{equation} 
		Combining these, we arrive at
		\begin{align*}
		f_\varepsilon''<\Omega[\widehat f](s)-d<\Omega[f_\varepsilon](s).
		\end{align*}
		That is, $f_\varepsilon$ defines $\gamma$ and $K$ satisfying the DEC.
	\end{proof}
	
	\begin{lem}[Bending lemma]\label{lem_bending_appendix}
		Let $(M,g,K)$ be an initial data set of the form $M=[a,\infty)\times \Sbb^2$, $a\ge 0$, 
		\begin{align*}
		g&=\d s^2+f(s)^2\gamma_*,\\
		K&=x'(f(s))\d s^2+x(f(s))f(s)\gamma_*,
		\end{align*}
		such that $\mu\ge \tau$ for some $\tau\in\R$, where $f$ is a smooth, positive function, $x$ is $C^1$, and $\gamma_*$ denotes the round metric on $\Sbb^2$.\newline
		For any $s_o>a$ such that $f'(s_o)>0$ there exists $\delta>0$ such that the following holds: There exists a smooth positive function $\widetilde{f}$ on $[s_o-\delta,\infty)$ such that $\widetilde{f}(s)=f(s)$ for all $s\in[s_o,\infty)$ and the initial data set $(\widetilde{M},\widetilde{g},\widetilde{K})$, with $\widetilde{M}=[s_o-\delta,\infty)$,
		\begin{align*}
		\widetilde{g}&=\d s^2+\widetilde{f}^2(s)\gamma_*,\\
		\widetilde{K}&=x'(\widetilde{f}(s))\d s^2+x(\widetilde{f}(s))\widetilde{f}(s)\gamma_*,
		\end{align*}
		has $\widetilde{\mu}\ge \tau$ such that $\widetilde{\mu}>\tau$ on $[s_o-\delta,s_o)$. \newline
		In addition, if $f(s_o)>c>0$ and $f''(s_o)>0$, then $\delta$ can be chosen such that $\widetilde{f}(s_o-\delta)>c$, $\widetilde{f}'(s_o-\delta)<f'(s_o)$.
	\end{lem}

	\begin{proof}
		Consider $\widetilde{f}(s)=f(\sigma(s))$ for some smooth reparametrisation $\sigma(s)$ to be determined. Then, by Lemma \ref{lem_mu}, we find that
		\begin{align*}
			\widetilde{\mu}(s)
			=&\,\frac{1}{f^2(\sigma(s))}-\frac{1}{f^2(\sigma(s))}\left((f'(\sigma(s))^2\dot{\sigma}(s)^2\right)\\
			&+\frac{2}{f(\sigma(s))}\left(f''(\sigma(s))\,\dot{\sigma}(s)^2+f'(\sigma(s))\overset{\cdot\cdot}{\sigma}(s)\right)\\
			&+\frac{1}{f^2(\sigma(s))}\left(x^2(f(\sigma(s)))+2x(f(\sigma(s)))\,x'(f(\sigma(s)))\,f(\sigma(s))\right)\\
			=&\,\mu(\sigma(s))+\frac{1}{f(\sigma(s))^2}\left((1-\dot{\sigma}(s)^2)\left(f'(\sigma(s))^2+2f(\sigma(s))\,f''(\sigma(s))\right)\right)\\
			&-\frac{2f'(\sigma(s))\overset{\cdot\cdot}{\sigma}(s)}{f(\sigma(s))}
		\end{align*}
		As $\mu\ge \tau$ by assumption, the claim thus follows if we can choose a constant $\delta>0$ and a smooth function $\sigma\colon[s_o-\delta,\infty)\to[a,\infty)$ such that
		\begin{enumerate}
			\item[(i)] $\sigma(s)=s$ on $[s_o,\infty)$,
			\item[(ii)] 
			\[
				2f(\sigma(s))\,f'(\sigma(s))\,\overset{\cdot\cdot}{\sigma}(s)<(1-\dot{\sigma}(s)^2)\left(f'(\sigma(s))+2f(\sigma(s))\,f''(\sigma(s))\right)
			\]
			on $[s_o-\delta,s_o)$.
		\end{enumerate}
		Consider the smooth function
		\[
			\sigma(s)=s-\int\limits_{s}^{s_o}\theta(t)\d t\text{ with }\theta(s)=
			\begin{cases}
				0&s\ge s_o,\\
				\exp\left(-(s-s_o)^{-2}\right)&s_o-\delta\le s<s_o,
			\end{cases}
		\]
		which satisfies (i) by definition. We now aim to show (ii):
		
		As $1\le \dot{\sigma}=1+\theta(s)\le 2$, we have $\sigma([s_o-\delta,s_o])\subseteq [s_o-2\delta,s_o]$. Hence, by continuity, there exists $0<\delta<\frac{1}{2}\btr{s_o-a}$ such that
		\begin{align*}
			&2f(\sigma(s))\,f'(\sigma(s))\ge f(s_o)\,f'(s_o)=c_1>0,\\
			&\max\left(0,f'(\sigma(s))^2+2f(\sigma(s))\,f''(\sigma(s)) \right)\le c_3,
		\end{align*}
		for all $s\in [s_o-\delta,s_o)$ for some constant $c_3\ge 0$, and where $c_1>0$ by assumption. As $\dot{\sigma}(s)>1$ on $[s_o-\delta,s_o)$, it thus suffices to show that
		\[
			\frac{c_3}{c_1}<\frac{\overset{\cdot\cdot}{\sigma}(s)}{1-\dot{\sigma}(s)^2}
		\]
		on $[s_o-\delta,s_o)$. Direct computation yields , 
		\begin{align*}
			\dot{\sigma}(s)^2&=1+2\theta(s)+\theta^2(s),\\
			\overset{\cdot\cdot}{\sigma}(s)&=-\frac{2}{\btr{s-s_o}^3}\theta(s).
		\end{align*}
		In particular,
		\[
			\frac{\overset{\cdot\cdot}{\sigma}(s)}{1-\dot{\sigma}(s)^2}=\frac{2}{(2+\theta(s))\btr{s-s_o}^3}\to \infty
		\]
		as $s\to s_o$. This implies (ii) for $\delta$ sufficiently small.
		
		Finally, assume that $f(s_o)>c>0$, $f''(s_o)>0$. By continuity, we can choose $\delta$ sufficiently small such that $\widetilde{f}(s_o-\delta)>c>0$. Additionally,
		\[
			\widetilde{f}''(s)=f''(\sigma(s))(1+2\theta(s)+\theta^2(s))-2f'(\sigma(s))\frac{\theta(s)}{\btr{s-s_o}^3},
		\]
		and as $\theta(s)\btr{s_o-s}^3\to 0$ as $s\to s_o$ we find that $\widetilde{f}''>0$ on $[s_o-\delta,\infty)$ if $f''(s_o)>0$ and $\delta$ is sufficiently small. Hence, $\widetilde{f}'$ is strictly increasing, which directly implies that
		\[
			\widetilde{f}'(s_o-\delta)<\widetilde{f}'(s_o)=f'(s_o).
		\]
		This finishes the proof.
	\end{proof}

\bibliographystyle{abbrv}
\bibliography{refs}

\end{document}